\documentclass{amsart}
\usepackage{amssymb, epic,eepic,epsfig,amsbsy,amsmath,amscd}
\usepackage[left=2cm,right=2cm,top=3cm,bottom=3.5cm,a4paper]{geometry}
\usepackage{color}
\usepackage{verbatim} 

\def\be#1\ee{\begin{equation}#1\end{equation}}
\theoremstyle{plain}

\newtheorem{Thm}{Theorem}

\newtheorem{proposition}[Thm]{Proposition}
\newtheorem{lemma}[Thm]{Lemma}
\newtheorem*{lemma*}{Lemma}
\newtheorem*{Theorem*}{Theorem}

\newtheorem{AThm}{Theorem}[section]
\newtheorem{Alemma}[AThm]{Lemma}
\newtheorem{Aproposition}[AThm]{Proposition}

\newtheorem{Acorollary}[AThm]{Corollary}

\theoremstyle{definition}

\newtheorem{remark}[Thm]{Remark}
\newtheorem*{remark*}{Remark}
\newtheorem*{example*}{Example}

\def\printname#1{
    \if\draft
        \smash{\makebox[0pt]{\hspace{-0.5in}
            \raisebox{8pt}{\tt\tiny #1}}}
    \fi
}

\newlength{\standardunitlength}
\catcode`\@=11
\long\def\@makecaption#1#2{%
     \vskip 10pt

\setbox\@tempboxa\hbox{
       \small\sf{\bfcaptionfont #1. }\ignorespaces #2}%
     \ifdim \wd\@tempboxa >\captionwidth {%
         \rightskip=\@captionmargin\leftskip=\@captionmargin
         \unhbox\@tempboxa\par}%
       \else
         \hbox to\hsize{\hfil\box\@tempboxa\hfil}%
     \fi}
\font\bfcaptionfont=cmssbx10 scaled \magstephalf
\newdimen\@captionmargin\@captionmargin=2\parindent
\newdimen\captionwidth\captionwidth=\hsize
\catcode`\@=12

\newcommand{\id}{\operatorname{id}}
\newcommand{\ord}{\operatorname{ord}}

\newcommand{\qbinom}[2]{\text{$\left[\begin{array}{c}#1\\ #2\end{array}
\right]$}}

\def\BZ{\mathbb Z}

\def\BQ{\mathbb Q}

\def\cT{\mathcal T}
\def\cM{\mathcal M}

\def\cR{\mathcal R}

\def\cL{\mathcal L}

\def\M{\mathcal M}
\def\N{{\mathbb N}}
\def\F{\mathcal F}

\def\R{\mathcal R}

\def\calS{\mathcal S}

\newcommand{\sn}{\operatorname{sn}}

\def\bp{{\mathbf p}}
\def\bn{{\mathbf n}}
\def\bj{{\mathbf j}}

\newcommand{\U}{\mathcal{U}}

\def\ve{\varepsilon}

\def\Z{\BZ}

\def\bk{\mathbf{k}}

\def\Q{\BQ}
\def\n{{\mathbf n}}
\def\bj{{\mathbf j}}

\def\ev{\mathrm{ev}}

\def\Habiro{\widehat{\Z[q]}}

\newcommand{\tF}{\tilde{F}}

\newcommand\nc{\newcommand}

\nc\FIG[3]{\begin{figure}
    \includegraphics[#3]{#1.eps}
    \caption{#2}
    \label{fig:#1}
    \end{figure}}

\nc\incl[2]{{\includegraphics[height=#1]{#2.eps}}}

\nc\zzzcolon {\colon\thinspace}
\nc\zzzvert {\ |\ }
\nc\simh{\underset{h}{\sim}}

\nc\trr{\triangleright}
\nc\et{\incl{1em}{bot0}}
\nc\ul{\underline}

\nc\modone {{\mathbf 1}}
\nc\modA {{\mathcal A}}
\nc\sfA{{\mathsf A}}
\nc\modb {{\mathsf b}}
\nc\modB {{\mathsf B}}
\nc\cB{\check{\modB }}
\nc\modC {{\mathcal C}}
\nc\modH {{\langle \HH\rangle }}
\nc\modh {{\mathsf h}}
\nc\modI {{\mathcal I}}
\nc\modJ {{\mathsf J}}
\nc\modk {{\mathbf k}}
\nc\modL {{\mathcal L}}
\nc\LL{{\mathsf L}}
\nc\HH{{\mathsf H}}
\nc\modM {{\mathcal M}}
\nc\sfP{{\mathsf P}}
\nc\modP {{\mathcal P}}
\nc\modQ {{\mathbb Q}}
\nc\modR {{\mathbb R}}
\nc\modr {v}
\nc\modS {{\mathcal S}}
\nc\modT {{\mathcal T}}
\nc\modV {{\mathsf V}}
\nc\modY {{\mathcal Y}}
\nc\modZ {{\mathbb Z}}

\nc\ad{{\operatorname{ad}}}
\nc\coad{{\operatorname{coad}}}
\nc\coadb{\underline{\coad}}
\nc\adb{\underline{\ad}}
\nc\coev{{\operatorname{coev}}}
\nc\pat{{\operatorname{\sf pat}}}
\nc\low{{\operatorname{low}}}
\nc\up{{\operatorname{up}}}

\nc\evdn{{\incl{.8em}{evdown}}}
\nc\cvdn{{\incl{.8em}{coevdown}}}
\nc\evup{{\incl{.8em}{evup}}}
\nc\cvup{{\incl{.8em}{coevup}}}

\nc\bH{{\underline H}}
\nc\bD{{\underline\Delta }}
\nc\bS{{\underline S}}
\nc\tS{{\tilde{S}}}

\nc\Ob{\operatorname{Ob}}

\begin{document}

\title[A unified  quantum  invariant]{A unified
 quantum SO(3) invariant\\[.1cm] for rational homology 3--spheres}

\author{Anna Beliakova}
\author{Irmgard B\"uhler}
\address{Institut f\"ur Mathematik, Universit\"at Zurich,
Winterthurerstrasse 190, 8057 Z\"urich, Switzerland}
\email{anna@math.uzh.ch, irmgard.buehler@math.uzh.ch}

\author{Thang Le}
\address{Department of Mathematics, Georgia Institute of Technology,
Atlanta, GA 30332--0160, USA }
\email{letu@math.gatech.edu}

\begin{abstract}
Given
 a rational homology 3--sphere  $M$  with $|H_1(M,\Z)|=b$ and a link 
$L$ inside $M$, colored by odd numbers, 
 we construct a unified invariant $I_{M,L}$
belonging to a modification of the Habiro ring where $b$ is inverted.
Our unified invariant dominates  the whole set of the $SO(3)$
  Witten--Reshetikhin--Turaev  invariants
of the pair $(M,L)$.
If $b=1$ and $L=\emptyset$, $I_M$ coincides with  Habiro's  invariant
of integral homology 3--spheres. For $b>1$, the unified invariant
defined by the third author is determined by $I_M$.
One of the applications are the new Ohtsuki series
 (perturbative expansions of $I_M$ at  roots of unity)
 dominating all quantum  $SO(3)$  invariants.


\end{abstract}

\maketitle


\section*{Introduction}

\subsection*{Background}
The $SU(2)$  Witten--Reshetikhin--Turaev (WRT)  invariant
is defined for any closed oriented 3--manifold $M$ and any
root of unity $\xi$ \cite{Tu}. Kirby and Melvin \cite{KM}
introduced the $SO(3)$ version of the invariant
$\tau_M(\xi)\in \Q(\xi)$
for roots of unity $\xi$ of odd order.
If the order  of $\xi$ is
prime, then by the results of Murakami \cite{Mu} (also Masbaum--Roberts
  \cite{MR}),
$\tau_M(\xi)$ is an algebraic
integer. This integrality result was the starting point for the
construction of finite type 3--manifold invariants, Ohtsuki series
 \cite{Ohtsukibook},
integral TQFTs, representations of the mapping class group
over $\Z[\xi]$ \cite{GM},
and  categorification of quantum 3--manifold invariants \cite{Kho}.
The proofs in \cite{Mu} and \cite{MR}
depend heavily on the arithmetic of $\Z[\xi]$ for a  root of unity $\xi$ of {\em prime} order
and do not extend to other  roots of unity.

 Is it true that $\tau_M(\xi)$ is always an algebraic integer
(belongs to $\Z[\xi]$),  even when the order of $\xi$ is not a
 prime? The positive answer to this question
was given first for {\em integral homology
spheres} by Habiro \cite{Ha}, and then for arbitrary 3--manifolds by
 the first and third author \cite{BL}, in connection with the
study of ``strong integrality''.

What Habiro proved for integral homology 3--spheres is actually
much stronger than integrality.
For any  integral homology
3--sphere $M$,
 Habiro \cite{Ha} constructed a {\em unified invariant} $J_M$
 whose evaluation  at any root of unity coincides with
the value of the  Witten--Reshetikhin--Turaev
invariant at that root.
Habiro's unified invariant $J_M$ is an element of the following
ring (Habiro's ring)
\[
\Habiro:=\lim_{\overleftarrow{\hspace{2mm}k\hspace{2mm}}}
\frac{\Z[q]}{
((q;q)_k)},  \qquad \text{ where} \quad (q;q)_k = \prod_{j=1}^k (1-q^j).
\]
Every element $f(q)\in \Habiro$ can be written as an infinite sum
\[
f(q)= \sum_{k\ge 0} f_k(q)\, (1-q)(1-q^2)...(1-q^k),
\]
with $f_k(q)\in \Z[q]$. When $q=\xi$, a root of unity,
 only a finite number of terms on the right hand side are not zero,
hence the right hand side gives a well--defined value, called the evaluation
$\ev_\xi(f(q))$.  Since $f_k(q)\in \Z[q]$,
 $\ev_\xi(f(q))\in \Z[\xi]$ is an algebraic integer.
 The fact that the unified invariant belongs to $\Habiro$
is stronger than just integrality of $\tau_M(\xi)$. We will  refer to it
as ``strong''  integrality.

The Habiro ring has beautiful
arithmetic properties.
Every element $f(q) \in \Habiro$ can be considered
 as a function whose domain is the set
 of roots of unity.
Moreover, there is a natural Taylor series for $f$ at every root
of unity.
Two elements $f,g \in \Habiro$ are the same if and
only if their Taylor series at a root of unity coincide.
In addition, each function $f(q) \in \Habiro$ is totally determined
by its values at, say,
infinitely many  roots of order $3^n,\, n\in \N$.
Due to these properties the Habiro ring is also called
a ring of ``analytic functions at roots of unity''
\cite{Ha}. Thus belonging to $\Habiro$ means that the
collection of the $SO(3)$ WRT invariants is far from a random
collection of algebraic integers; together they form a nice function.

Perturbative expansion at 1 of WRT invariants for rational homology 3--spheres
  was first constructed by
Ohtsuki in the case when the order of the quantum parameter $\xi$ is
prime \cite{Oh}.
General properties of the Habiro ring imply that for any integral homology
3--sphere $M$,
the Taylor expansion of the unified invariant $J_M$
at $q= 1$ coincides with the Ohtsuki series
and dominates WRT invariants of $M$ at all roots of unity (not only of prime order).

To generalize Habiro's results to
rational homology 3--spheres, new ideas and techniques are required.
Strong integrality of quantum invariants for rational homology 3--spheres
was studied in \cite{Le} and \cite{BL}. Among other things,
in \cite{Le}  a unified invariant was constructed for the case when
 the order $r$ of the quantum parameter $\xi$
is coprime with $b$. In \cite{BL}, it was proved that
 for any 3--manifold $M$ (not necessary  a rational homology 3--sphere),
 the $SO(3)$
WRT invariant $\tau_M(\xi)$ is always an
algebraic integer, i.e. $\tau_M(\xi)\in \Z[\xi]$  with no restriction on
 the order of $\xi$ at all.
 There we used a
(2nd order) Laplace transform method \cite{BBL} and a difficult
identity of Andrews \cite{And} in $q$--calculus, generalizing those of
 Rogers--Ramanujan.

Thus, although we have had integrality of all SO(3) WRT invariants,
we still
lacked a ``strong integrality'' for the case
when $(r,b) \neq 1$. This is the main object of this paper.

In this paper
 we  will generalize Habiro's
 construction of the unified invariant
 to all  rational homology
3--spheres.
Our new unified invariant $I_M$
dominates SO(3)  WRT invariants also in the case
when  the order $r$
of the quantum parameter is not coprime with $b=|H_1(M,\Z)|$.
Although this includes the case $(r,b)=1$ of \cite{Le},
the ring our invariant belongs to
 is simpler than the one obtained in \cite{Le} and \cite{BL}. In particular,
 we don't need any fractional power of $q$.
We show that the Taylor expansion of our unified invariant
at a root of unity of order $c$ (new Ohtsuki series)
dominates all WRT invariants with $r=cl$ and $(l,b)=1$.

For rational homology 3--spheres the universal finite type
invariant was constructed by
 Le, Murakami and Ohtsuki  \cite{LMO}. It
 determines Ohtsuki series and, hence,
$\{\tau_M(\xi)\,|\,(\ord(\xi),b)=1\}$
 \cite{Le}. An interesting open question is whether
 the Le--Murakami--Ohtsuki invariant  determines
$I_M$.


\subsection*{Results} The WRT or quantum $SO(3)$ invariant $\tau_{M,L}(\xi)$ is defined for a pair of a closed 3--manifold $M$ and a link $L$ in it,
with link components colored by integers. Here $\xi$ is a root
of unity of odd order. We will recall the definitions in Section \ref{defs}.

Suppose
$M$ is a rational homology
3--sphere, i.e.
$|H_1(M,\Z)|:={\rm card}\,H_1 (M,\Z) < \infty$.
There is a unique decomposition
 $ H_1(M,\Z)=\bigoplus_{i}
\Z/{b_{i}\Z}$, where each $b_i$ is a prime power.
We renormalize the $SO(3)$  WRT invariant of
the pair  $(M, L)$ as follows:
\begin{equation}
\tau'_{M, L}(\xi)=\frac{\tau_{M,L}(\xi)}
{\prod\limits_{i}\;\,
\tau_{L(b_{i},1)}(\xi)}\; ,
\label{0910}
\end{equation}
where $L(b,a)$ denotes the $(b,a)$ lens space. We will see that
$\tau_{L(b,1)}(\xi)$ is always nonzero.

 For any positive integer $b$, we  define the
cyclotomic completion ring $\R_b$ to be
\be
\label{ab} \R_b:=\lim_{\overleftarrow{\hspace{2mm} k\hspace{2mm}}}
\frac{\Z[1/b][q]}{
\left((q;q^2)_k\right)}, \qquad \text{where} \quad
(q;q^2)_k = (1-q)(1-q^3) \dots (1-q^{2k-1}).
\ee
For any    $f(q)\in \R_b$ and a root of unity $\xi$ of {\em odd} order,
the evaluation $\ev_\xi (f(q)):= f(\xi)$ is well--defined.
Similarly, we put
$$\calS_b :=\lim_{\overleftarrow{\hspace{2mm}k\hspace{2mm}}}
\frac{\Z[1/b][q]}
{((q;q)_k)}\; .$$
Here the evaluation at any root of unity is well--defined. For odd $b$,
there is  a natural embedding
$\calS_b\to \R_b$, see Section \ref{cyc}.

Let us denote by $\M_b$ the set of rational homology 3--spheres such that
$|H_1(M,\Z)|$ divides $b^n$ for some $n$.
The main result of this paper is the following.
\begin{Thm}\label{main}
Suppose the components of a framed oriented link $L \subset M$ have odd colors, and $M\in \cM_b$. Then
there exists an invariant $I_{M,L} \in \R_b$,
such that for any root of unity $\xi$ of odd order
$$\ev_\xi(I_{M,L})=\tau'_{M,L}(\xi)\, .$$
In addition, if $b$ is odd, then
$I_{M,L}\in \calS_b$.
\end{Thm}

If $b=1$ and $L$ is the empty link, $I_{M}$ 
coincides with Habiro's unified invariant $J_M$.

The proof of Theorem \ref{main} uses  the Laplace transform method and Andrew's identity.
We also construct a Frobenius type isomorphism to get rid of the
formal fractional power of $q$ that appeared in \cite{Le}, \cite{BL}.
As a byproduct, we generalize the deep integrality result of Habiro
(Theorem 8.2 in \cite{Ha}), underlying the construction
of the unified invariants,  to a union of an algebraically split link
with any odd colored one.

The rings $\R_b$ and $\calS_b$ have properties
 similar  to those of the Habiro ring.
An element $f(q) \in \R_b$ is totally determined by the values at many
infinite sets of roots of unity (see Section \ref{cyc}),
one special case is the following.

\begin{proposition}\label{main-cor} Let $p$ be an odd prime
not dividing $b$ and $T$ the set of all integers of the form
$p^k b'$ with $k\in \N$ and $b'$ any odd divisor of
$b^n$ for some $n$. Any element $f(q) \in \R_b$, and hence also
$\{\tau_M(\xi)\}$, is totally determined by the values at roots of
unity with orders in $T$.
\end{proposition}

The Ohtsuki series \cite{Oh,Le3},
originally defined through some arithmetic
congruence property of the $SO(3)$ invariant,
 can be identified with the Taylor expansion of $I_M$ at
$q=1$ \cite{Ha, Le}. We will also investigate the Taylor expansions
of $I_M$ at roots of unity and show that these Taylor expansions satisfy
congruence relations similar to the original definition of the Ohtsuki
series, see Section \ref{Oh-strategy}.

\subsection*{Plan of the paper}
In Section \ref{defs} we  recall  known results
and definitions. In the next section we explain
the strategy of our proof of Theorem \ref{main}.
In Sections \ref{cyc} and \ref{map-frob},
we develop properties of cyclotomic completions
of polynomial rings.
New Ohtsuki series are discussed in Section \ref{Oh-strategy}.
The unified invariant of  lens spaces, needed for the diagonalization,
is defined in Section \ref{diag}.
The main technical
result
of the paper based on Andrew's identity is proved in Section \ref{laplace}.
The Appendix is devoted to the proof of the generalization
of  Habiro's integrality theorem.

\subsection*{Acknowledgments} The authors would like to thank
Kazuo Habiro and Christian Krattenthaler  for helpful
remarks and stimulating conversations.


\section{Quantum (WRT) invariants}
 \label{defs}

\subsection{Notations and conventions}
We will consider $q^{1/4}$ as a free parameter. Let
\[
\{n\} = q^{n/2}-q^{-n/2},
 \quad  \{n\}!=
\prod_{i=1}^n \{i\} ,\quad  [n] =\frac{\{n\}}{\{1\}}, \quad
\qbinom{n}{k} = \frac{\{n\}!}{\{k\}!\{n-k\}!}.
\]
We denote the set $\{1,2,3,\ldots\}$ by
$\N$.
We also use the following notation from $q$--calculus:
$$ (x;q)_n := \prod_{j=1}^n (1-x q^{j-1}).$$
Throughout this paper,
  $\xi$ will be  a primitive root
of unity of {\em odd} order $r$ and $e_n:=\exp(2\pi I/n)$.

All 3--manifolds in this paper are supposed to be  closed and
oriented. Every link in a 3--manifold is framed, oriented, and has
components ordered.

In  this paper,
 $L\sqcup L'$ denotes a framed link in $S^3$ with
disjoint sublinks $L$ and $L'$, with $m$ and $l$ components, respectively.
Surgery along the framed link $L$ transforms $(S^3,L')$ into  $(M,
L')$. We use the same notation $L'$ to denote the link in $S^3$ and
the corresponding one in $M$.


\subsection{The colored Jones polynomial}
\newcommand{\RR} {\mathbf R}
Suppose $L$ is a framed, oriented link
in $S^3$ with $m$ ordered components.
For positive integers $n_1,\dots,n_m$, called the colors of $L$, one can define
the quantum invariant $J_L(n_1,\dots,n_m)\in \Z[q^{\pm 1/4}]$,
known as the colored
Jones polynomial of $L$ (see e.g. \cite{Tu,MM}).
Let us recall here a few well--known formulas.
For the unknot $U$ with 0 framing one has
\begin{equation} J_U(n) = [n]. \label{unknot}
\end{equation}
 If $L_1$ is obtained from $L$
by increasing the framing of the $i$th component by 1, then
\begin{equation}\label{framing}
J_{L_1}(n_1,\dots,n_m) = q^{(n_i^2-1)/4} J_{L}(n_1,\dots,n_m).
\end{equation}
If all the colors $n_i$ are odd, then $J_{L}(n_1,\dots,n_m) \in \Z[q^{\pm 1}]$.


\subsection{Evaluation and Gauss sums}
For each root of unity $\xi$ of odd order $r$,
we define the evaluation map
$\ev_\xi$ by replacing $q$ with $\xi$.


Suppose
  $f(q;n_1,\dots,n_m)$ is a function
of variables $q^{\pm 1}$ and integers $n_1,\dots,n_m$.
In quantum topology, the following sum plays an important role
\[
{\sum_{n_i}}^\xi f := \sum_{\substack{0< n_i< 2r\\ n_i \text{ odd}}}
\ev_\xi  f(q; n_1,\dots, n_m)
\]
where in the sum  all the $n_i$ run over the set of {\em odd} numbers
between $0$ and $2r$.

In particular, the following  variation  of the Gauss sum
\[
\gamma_b(\xi):= {\sum_{n}}^\xi q^{b\frac{n^2-1}{4}}
\]
is well--defined, since for odd $n$, $4\mid n^2-1$.
It is known that, for odd $r$, $|\gamma_b(\xi)|$ is never 0.

\subsection{Definition of  the WRT invariant}
\newcommand{\tL}{\tilde L}
Suppose the components of $L'$ are colored by fixed integers
$j_1,\dots,j_l$.
 Let
\[
 F_{L\sqcup L'}(\xi):= {\sum_{n_i}}^\xi\;
\left \{  J_{L\sqcup L'}(n_1,\dots,n_m, j_1,\dots,j_l)\prod_{i=1}^m [n_i]\right \}.\]

An important special case is when $L=U^b$, the unknot with framing
$b \neq 0$, and $L'=\emptyset$. In that case $F_{U^{b}}(\xi)$ can be calculated
using the Gauss sum and is nonzero, see  Section \ref{diag}
below.

 Let $\sigma_+ $ (respectively $\sigma_-$) be the number of
positive (negative) eigenvalues of the linking matrix of $L$.
 Then the quantum
$SO(3)$ invariant of the pair $(M, L')$ is defined by (see e.g. \cite{KM,Tu})
\begin{equation}
\tau_{M,L'}(\xi) =
\frac{F_{L\sqcup L'}(\xi)}{(F_{U^{+1}}(\xi))^{\sigma_+}\,
(F_{U^{-1}}(\xi))^{\sigma_-} }\, .
\label{def_qi}
\end{equation}
The invariant $\tau_{M,L'}(\xi)$ is multiplicative with
respect to the connected sum.

For example, the $SO(3)$ invariant of the lens space $L(b,1)$,
obtained by surgery along $U^b$,  is
\begin{equation} \tau_{L(b,1)} (\xi)= \frac{ F_{U^b}(\xi)}{F_{U^{\sn(b)}}(\xi) },
\label{2005}
\end{equation}
were $\sn(b)$ is the sign of the integer $b$.

Let us focus on the special case when the linking matrix of
$L$ is diagonal, with $b_1, b_2, \dots, b_m$ on the diagonal.
Assume each $b_i$ is a power of a prime up to sign.
Then $H_1(M,\Z) = \oplus_{i=1}^m \Z/|b_i|$, and
$$\sigma_+ = {\rm card}\, \{ i\mid b_i >0\}, \quad \sigma_- =
{\rm card}\, \{ i \mid b_i < 0\}.$$
Thus from the definitions \eqref{def_qi}, \eqref{2005} and \eqref{0910} we have
\begin{equation}
 \tau'_{M,L'}(\xi) = \left( \prod_{i=1}^m  \tau'_{L(b_i,1)}(\xi)
 \right)\,
\frac{F_{L\sqcup L'}(\xi)}
{\prod_{i=1}^m F_{U^{b_i}}(\xi) }
\,  ,
\label{0077}
\end{equation}
with
$$\tau'_{L(b_i,1)}(\xi)= \frac{\tau_{L(b_i,1)}(\xi)}{\tau_{L(|b_i|,1)}(\xi)}\, .$$

\subsection{Habiro's cyclotomic expansion of the colored Jones
polynomial}
Recall that  $L$ and $L'$ have $m$ and $l$ components, respectively. Let us color $L'$ by fixed $\bj=(j_1,\dots,j_l)$ and vary the colors
$\bn=(n_1,\dots,n_m)$ of $L$.

For non--negative integers $n,k$ we define
$$ A(n,k) := \frac{\prod^{k}_{i=0}
\left(q^{n}+q^{-n}-q^i -q^{-i}\right)}{(1-q) \, (q^{k+1};q)_{k+1}}.$$
For $\bk=(k_1,\dots,k_m)$ let
$$ A(\bn,\bk):= \prod_{j=1}^m \; A(n_j,k_j).$$
Note that $A(\bn,\bk)=0$ if $k_j \ge n_j$ for some index $j$. Also

$$ A(n,0)= q^{-1} J_U(n)^2.$$

The colored Jones polynomial $J_{L\sqcup L'}
(\n, \bj)$, when $\bj$ is fixed, can be repackaged into the invariant $C_{L\sqcup L'}
(\bk, \bj)$ as stated in the following theorem.
\begin{Thm}\label{GeneralizedHabiro} 
 Suppose $L\sqcup L'$ is a link in $S^3$, with $L$ having zero linking matrix.
Assume the components of $L'$ have fixed {\em odd} 
colors $\bj = (j_1,\dotsm j_l)$.
Then  there are invariants
\begin{equation}\label{Jones2}
C_{L\sqcup L'}(\bk,\bj) \in \frac{(q^{k+1};q)_{k+1}}{(1-q)}
\,\,\BZ[q^{\pm 1}] ,\quad \text{where  $k=\max\{k_1,\dots, k_m\}$}
\end{equation}
such that for every $\bn =(n_1,\dots, n_m)$
\begin{equation}\label{Jones}
 J_{L\sqcup L'}
(\n, \bj)  \,  \prod^m_{i=1}\; [n_i] = \sum_{0\le k_i \le n_i-1}
C_{L\sqcup L'}(\bk,\bj)\;
  A(\bn, \bk).
\end{equation}
\end{Thm}

 When $L'=\emptyset$, this is
 Theorem 8.2 in \cite{Ha}.
This generalization, essentially also due to Habiro, can be proved similarly as in \cite{Ha}.
 For completeness we give
a proof
 in the Appendix. Note that the existence of $C_{L\sqcup L'}(\bk,\bj)$ as
rational functions in $q$ satisfying \eqref{Jones} is
 easy to establish. The difficulty here is to show the integrality
of \eqref{Jones2}.

 Since $A(\bn, \bk) =0$ unless $ \bk < \bn$, in the sum on the right
hand side of \eqref{Jones} one can assume that $\bk$ runs over the set
 of all $m$--tuples $\bk$ with non--negative integer components. We will use
this fact later.



\section{Strategy of the proof of the main theorem }
\label{strategy}
Here we give the proof of Theorem \ref{main} using
technical results that will be proved later.

As before, $ L\sqcup L'$ is a framed link in $S^3$ with
disjoint sublinks $L$ and $L'$, with $m$ and $l$ components, respectively.
Assume that $L'$ is colored by fixed $\bj=(j_1,\dots, j_l)$, with $j_i$'s odd.
Surgery along the framed link $L$ transforms $(S^3,L')$ into  $(M,
L')$.
 We will define $I_{M,L'}\in \cR_{b}$, such that
\begin{equation}
\tau'_{M,L'}(\xi)\;=\; \ev_{\xi}\left(I_{M,L'}\right)
\label{0080}
\end{equation}
for any root of unity $\xi$ of odd order.
This unified invariant is multiplicative with respect
to the connected sum.

The following observation is important.
By Proposition \ref{main-cor}, there is {\em at most one} element $f(q)\in \R_b$  such that for every root $\xi$ of odd order one has
$$ \tau'_{M,L} (\xi) = \ev_\xi\left( f(q)\right).$$
That is, if  we can find such an element,
 it is unique, and we put
$I_{M,L'} := f(q)$.

\subsection{Laplace transform}
The following is the main technical result of the paper. A proof will be
given in Section \ref{laplace}.
\begin{Thm} Suppose $b=\pm 1$ or $b= \pm p^l$ where $p$  is a  prime and $l$ is
positive. For  any non--negative integer $k$,
there exists an element $Q_{b,k} \in \R_b $ such that
for every root $\xi$ of odd order $r$ one has
\[
\frac{{\sum\limits_n}^\xi \, q^{b\frac{n^2-1}{4}} A(n,k) }{F_{U^b}(\xi)}
= \ev_\xi (Q_{b,k}).
\]
\label{0078}
In addition, if $b$ is odd, $Q_{b,k} \in \calS_b $.
\end{Thm}


\subsection{Definition of the unified invariant: diagonal case} \label{2501}
Suppose that the linking number between any two components of $L$
is 0, and  the framing on components of $L$ are $b_i=\pm p_i^{k_i}$ for
$i=1,\dots, m$, where each $p_i$ is prime or 1.
Let us denote the link $L$ with all framings switched to zero by $L_0$.

Using \eqref{Jones}, taking into account the framings $b_i$'s, we have
\[
J_{L\sqcup L'}(\bn,\bj)\prod_{i=1}^m [n_i] = \sum_{\bk\ge 0} C_{L_0 \sqcup L'}
(\bk,\bj) \, \prod_{i=1}^m q^{b_i \frac{n_i^2-1}{4}} A(n_i,k_i).
\]
By the definition of $F_{L\sqcup L'}$, we have
\[ F_{L\sqcup L'}(\xi)= \sum_{\bk \ge 0}
\ev_\xi(C_{L_0 \sqcup L'}(\bk,\bj)) \, \prod_{i=1}^m
{\sum_{n_i}}^\xi \,  q^{b_i \frac{n_i^2-1}{4}} A(n_i,k_i).
\]
From \eqref{0077} and  Theorem \ref{0078}, we get
\[
\tau'_{M,L'}(\xi) = \ev_\xi \left \{  \prod_{i=1}^m  I_{L(b_i,1)} \,
 \sum_{\bk} C_{L_0 \sqcup L'}(\bk,\bj) \, \prod_{i=1}^m  Q_{b_i,k_i} \right \},
\]
where the unified invariant of the lens space $I_{L(b_i,1)}\in \R_b$,
with $\ev_\xi(I_{L(b_i,1)})=\tau'_{L(b_i,1)}(\xi)$, exists
 by  Lemma \ref{0079} below.
Thus if we define
\[
I_{(M,L')}:= \prod_{i=1}^m  I_{L(b_i,1)} \, \sum_{\bk}
C_{L_0 \sqcup L'}(\bk,\bj) \, \prod_{i=1}^m  Q_{b_i,k_i}\, ,
\]
then \eqref{0080} is satisfied. By Theorem \ref{GeneralizedHabiro},
$C_{L_0 \sqcup L'}(\bk,\bj)$ is divisible by $(q^{k+1};q)_{k+1}/(1-q)$, which is
divisible by $(q;q)_k$, where $k = \max k_i$. It follows that
$I_{(M,L')} \in \R_b$. In addition, if $b$ is odd,
then $I_{(M,L')} \in \calS_b$.

\subsection{Diagonalization using lens spaces} The general
case reduces to the diagonal case by
the well--known trick of diagonalization using lens spaces. We say that
$M$ is {\em diagonal} if it can be obtained from $S^3$ by surgery along a framed link
$L$  with diagonal linking matrix, where  the diagonal entries are of the
form $\pm p^k$ with $p=0,1$ or a prime.
The following lemma was proved in \cite[Proposition 3.2 (a)]{Le}.

\begin{lemma} For every rational homology sphere $M$,
 there are lens spaces $L(b_i,a_i)$ such that the connected
sum of $M$ and these
lens spaces is diagonal. Moreover, each $b_i$ is a prime power
 divisor of $|H_1(M,\Z)|$.
\label{diagonalization}
\end{lemma}

To define  the unified invariant for a general rational homology sphere $M$,
one first adds to $M$ lens spaces to get a diagonal $M'$, for which
the unified invariant $I_{M'}$ had been defined in Subsection \ref{2501}. Then
  $I_M$ is the quotient of $I_{M'}$ by the unified invariants of the lens
 spaces. But unlike the simpler
 case
of \cite{Le}, the unified invariant of lens spaces are {\em not} invertible
 in general.
To overcome this difficulty we insert knots in lens spaces and split the
 unified invariant into different components. This
 will be explained in the remaining part of this section.

\subsection{Splitting of the invariant} Suppose $p$ is a prime divisor of $b$,
 then it's clear that $\R_p \subset \R_b$.

In Section \ref{cyc} we will see that there is a decomposition

\newcommand{\bz} {{\bar 0}}
\[
\R_b = \R_{b}^{p,0} \times \R_{b}^{p,\bz},
\]
with canonical projections $\pi^p_0: \R_b \to \R_{b}^{p,0}$ and
$\pi^p_\bz:\R_b \to \R_{b}^{p,\bz}$. If $f\in \R_{b}^{p,0}$ then $\ev_\xi(f)$
 can be defined when the order of $\xi$ is coprime with $p$; and in this case
$ \ev_\xi(g) = \ev_\xi(\pi^p_0(g))$
for every $g\in \R_b$.

 On the other hand, if $f\in \R_{b}^{p,\bz}$ then $\ev_\xi(f)$
can be defined when the order
 of $\xi$ is divisible by $p$, and one has $ \ev_\xi(g) = \ev_\xi(\pi^p_\bz(g))$
for every $g\in \R_b$.

It also follows from the definition that $\R_p^{p,\ve}
\subset \R_b^{p,\ve}$ for $\ve = 0$ or $\bz$.

For $\calS_b$, there exists a
completely analogous decomposition.
For any odd divisor $p$ of $b$,
an element $ x \in \R_b$ (or $\calS_b$)
determines and is totally determined  by the
 pair $(\pi^p_0(x), \pi^p_\bz(x))$. If $p=2$ divides $b$, then for any
$x\in \R_b$, $x=\pi^p_0(x)$.

Hence, to define $I_M$ it is enough to fix
$I^0_M= \pi^p_0(I_M)$ and $I^\bz_M=\pi^p_\bz(I_M)$. The first
part $I^0_M= \pi^p_0(I_M)$, when $b=p$,
was  defined in \cite{Le} (up to normalization), where the
third author considered the case
when the order of roots of unity is coprime with $b$. We will give a
self--contained definition of $I^0_M$,
and show that it is coincident (up to normalization) with the one introduced in \cite{Le}.

\subsection{Lens spaces}

Suppose $b,a,d$ are integers with $(b,a)=1$ and $b\neq 0$. Let $M(b,a;d)$ be the pair
 of a lens space $L(b,a)$ and a knot $K\subset L(b,a)$, colored by
$d$, as described in
Figure \ref{figure:LensSpaceKnot}. \begin{figure}[htbp]
\begin{center}
\input{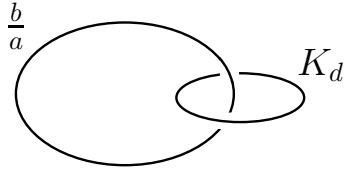}
\caption{The lens space  $(L(b,a),K_d)$ is obtained by
$b/a$ surgery on the first component of the Hopf link, the second
component is the knot $K$ colored by $d$.}
\label{figure:LensSpaceKnot}
\end{center}
\end{figure}
Among these pairs we want to single out some
whose quantum invariants are invertible.

For $\ve\in \{0,\bz\}$, let $M^\ve(b,a) := M(b,a;d(\ve))$,
where $d(0):=1$ and $d(\bz)$ is the smallest odd positive
integer such that $|a|d(\bz) \equiv 1 \pmod {b}$. Note that if $|a|=1$,
$d(0)=d(\bz)=1$.

It is known that if the color of a
link component is $1$, then the component can be removed from the link
without  affecting the value of quantum invariants. Hence
$$ \tau_{M(b,a;1)} = \tau_{L(b,a)}.$$

\begin{lemma}  Suppose $b=\pm p^{l}$ is a prime power.
For $\ve \in \{0,\bz\}$, there exists an invertible
invariant $I^\ve_{M^\ve(b,a)} \in \R^{p,\ve}_p$ such that
$$\tau'_{M^\ve(b,a)}(\xi)\;=\; \ev_{\xi}\left(I^\ve_{M^\ve(b,a)}\right)$$
where $\ve=0$ if the order of $\xi$ is not divisible by $p$,
and $\ve=\bz$ otherwise.
Moreover, if $p$ is odd, then
$I^\ve_{M^\ve(b,a)}$ belongs to and is  invertible in $\calS^{p,\ve}_p$.
\label{0079}
\end{lemma}
A proof of Lemma \ref{0079} will be given in Section \ref{diag}.

\subsection{Definition of the unified invariant: general case}
Now suppose $(M,L')$ is an arbitrary pair of a rational homology 3--sphere
 with a link
$L'$ in it colored by odd numbers $j_1,\dots, j_l$.
Let $L(b_i,a_i)$ for  $i=1,\dots, m$  be the
lens spaces of Lemma \ref{diagonalization}.
We use induction on $m$. If $m=0$, then $M$ is diagonal and $I_{M,L'}$
 has been defined in Subsection \ref{2501}.

Since $(M,L') \# M(b_1,a_1;d)$ becomes diagonal after adding
$m-1$ lens spaces,
the unified invariant
 of $(M,L') \# M(b_1,a_1;d)$ can be defined
by  induction, for any odd integer $d$.
In particular, one can define $I_{M^\ve}$, where
$M^\ve := (M,L') \# M^\ve(b_1,a_1)$. Here $\ve = 0$ or $\ve =\bz$ and
 $b_1$ is a power of a prime $p$ dividing $b$.
It follows that  the components $\pi^p_\ve(I_{M^\ve}) \in \R^{p,\ve}_b$  are
 defined.

By Lemma \ref{0079}, $I^\ve_{M^\ve(b_1,a_1)}$ is
defined and  invertible. Now we put
$$ I^\ve_{M,L'} := I^\ve_{M^\ve} \cdot  (I^\ve_{M^\ve(b_1,a_1)})^{-1}.$$
It is easy to see that $I_{M,L'}:= (I^0_{M,L'}, I^\bz_{M,L'})$
satisfies \eqref{0080}.
This completes the construction of $I_{M,L'}$.
It remains to prove Lemma \ref{0079}
and Theorem \ref{0078}.



\section{Cyclotomic completions of polynomial rings}\label{cyc}
In this section we adapt the results of Habiro on cyclotomic completions of polynomial rings \cite{Ha1} to our rings.
\subsection{On cyclotomic polynomial}
Recall that $e_n := \exp(2\pi I/n)$ and denote by
$\Phi_n(q)$ the cyclotomic polynomial
\[
\Phi_n(q) = \prod_{\substack{(j,n)=1\\0<j<n}} (q - e_n^j).
\]
The degree of $\Phi_n(q)\in \Z[q]$ is given by the Euler function $\varphi(n)$.
Suppose $p$ is a prime and $n$ an integer. Then (see e.g. \cite{Na})
\begin{equation} \Phi_n(q^p)= \begin{cases}    \Phi_{np}(q)  & \text{ if } p \mid n \\
\Phi_{np}(q) \Phi_n(q)  & \text{ if } p \nmid n.
\end{cases}
\end{equation}
It follows that  $\Phi_n(q^p)$ is always divisible by $\Phi_{np}(q)$.

The ideal of $\Z[q]$ generated by $\Phi_n(q)$ and $\Phi_m(q)$ is well--known,
see e.g. \cite[Lemma 5.4]{Le}:

\begin{lemma}$\text{ }$
\begin{itemize}
\item[(a)] If $\frac{m}{n} \neq p^e$ for any  prime $p$ and any integer $e\neq 0$, then
$(\Phi_n)+ (\Phi_m)=(1)$ in $\Z[q]$.

\item[(b)] If $\frac{m}{n} = p^e$ for a prime $p$ and some integer $e \neq 0$, then $(\Phi_n)+ (\Phi_m)=(1)$ in $\Z[1/p][q]$.
\label{0911}
\end{itemize}
\end{lemma}

Note that in a commutative ring $R$,  $(x) + (y) =(1)$ if and only if
$x$ is invertible in $R/(y)$. Also $(x) + (y) =(1)$ implies $(x^k) +
(y^l) =(1)$ for any integers $k,l \ge 1$.

%


\subsection{Habiro's results}  Let us summarize some of Habiro's
results on cyclotomic completions of polynomial rings \cite{Ha1}.
Let  $R$ be a commutative integral domain of characteristic zero
and $R[q]$  the polynomial
ring over $R$.
For any
 $S\subset \N$, Habiro defined  the $S$--cyclotomic completion
 ring $R[q]^S$ as follows:
\be\label{rs} R[q]^S:=\lim_{\overleftarrow{f(q)\in
\Phi^*_S}} \;\;\frac{R[q]}{(f(q))} \ee
 where $\Phi^*_S$ denotes the multiplicative
set in $\Z[q]$ generated by $\Phi_S=\{\Phi_n(q)\mid n\in S\}$
  and directed with respect to the
divisibility relation.

For example, since the sequence $(q;q)_n$, $n\in \N$,
 is cofinal to $\Phi^*_\N$, we have
\be\label{cofinal}
\Habiro\simeq\Z[q]^\N.
\ee

 Note that if $S$ is finite, then
$R[q]^S$ is identified with the $(\prod \Phi_S)$--adic completion of $R[q]$.
In particular,
\[
R[q]^{\{1\}}\simeq R[[q-1]], \quad
R[q]^{\{2\}}\simeq R[[q+1]].
\]
Suppose $S' \subset S$, then $\Phi^*_{S'}\subset \Phi^*_S$, hence
there is a natural map
\[
\rho^R_{S, S'}: R[q]^S \to R[q]^{S'}.
\]

Recall important results concerning  $R[q]^S$ from \cite{Ha1}. Two
positive integers $n, n'$ are called {\em adjacent} if
$n'/n=p^e$ with a nonzero
$e\in \Z$ and  a  prime $p$, such that
 the ring $R$ is $p$--adically separated, i.e. $\bigcap_{n=1}^\infty (p^n) =0$ in $R$.
A set of positive integers is
{\em $R$--connected} if for any two distinct elements $n,n'$ there is a
sequence $n=n_1, \,n_2, \dots,\, n_{k-1},\, n_k= n'$ in the set, such that
any two consecutive numbers of this sequence are adjacent.
Theorem 4.1 of \cite{Ha1}
says that if $S$ is $R$--connected,
 then for any subset $S'\subset S$
the natural map $ \rho^R_{S,S'}: R[q]^S \hookrightarrow R[q]^{S'}$ is an
embedding.

If $\zeta$ is a root of unity of order in $S$, then for every $f(q)\in R[q]^S$
the evaluation $\ev_\zeta(f(q))\in R[\zeta]$ can be defined by sending
$q\to\zeta$.
For a set $\Xi$ of roots of unity whose orders form a subset
$\cT\subset S$, one defines the evaluation
\[
\ev_\Xi: R[q]^S \to \prod_{\zeta \in \Xi} R[\zeta].
\]
Theorem 6.1 of \cite{Ha1} shows that if
$R\subset \Q$, $S$ is $R$--connected, and there
exists $n\in S$ that is adjacent to infinitely many elements in
$\cT$, then $\ev_\Xi$ is injective.


\subsection{Taylor expansion} Fix a natural number $n$, then we have
$$ R[q]^{\{n\}} =
\lim_{\overleftarrow{\hspace{2mm}k\hspace{2mm}}}
\frac{R[q]}{(\Phi^k_n(q))}\; .$$
Suppose $ \Z \subset R \subset \Q$, then the natural algebra homomorphism

$$ h: \frac{R[q]}{(\Phi^k_n(q))} \to \frac{R[e_n][q]}{((q-e_n)^k)}$$
is injective, by Proposition \ref{h_kInjektive} below. Taking the inverse limit, we see that there is a natural injective algebra homomorphism

$$ h : R[q]^{\{n\}} \to R[e_n][[q-e_n]].$$

Suppose $n \in S$. Combining $h$ and $\rho_{S, \{n\}}: R[q]^S \to R[q]^{\{n\}}$, we get an algebra map

\newcommand{\TT}{{\mathfrak t}}

$$ \TT_n: R[q]^S \to R[e_n][[q-e_n]].$$
If $f\in R[q]^S$, then $\TT_n(f)$ is called the Taylor expansion of $f$ at $e_n$.

\subsection{Splitting of $\calS_p$ and evaluation}
For every integer $a$, we put $\N_a := \{ n \in \N \mid (a,n)=1\}$.

Suppose $p$ is a prime. Analogously to \eqref{cofinal},
we have
$$\calS_p \simeq \Z[1/p][q]^\N\,.$$

Observe that $\N$ is not $\Z[1/p]$--connected. In fact one has
 $\N =\amalg_{j=0}^\infty \; p^j \N_p$, where each $p^j\N_p$ is
$\Z[1/p]$--connected.
Let us define
\[
\calS_{p,j}:= \Z[1/p][q]^{p^j \N_p}.
\]
Note that for every $f \in \calS_p$, the evaluation $\ev_\xi(f)$ can be
 defined for every root $\xi$ of unity. For $f\in \calS_{p,j}$, the evaluation
$\ev_\xi(f)$ can be defined  when $\xi$ is a root of unity of order
in $p^j\N_p$.

\begin{proposition} For every prime $p$ one has
\begin{equation} \calS_p\simeq \prod_{j=0}^\infty \calS_{p,j}.
\label{0912}\end{equation}

\end{proposition}

\begin{proof}
Suppose $n_i \in p^{j_i}\N_p$ for $ i=1,\dots, m$, with  distinct $j_i$'s.
Then $n_{i}/n_{s}$, with $i \neq s$, is either not a power of a prime or
 a non--zero power of $p$, hence
 by
Lemma \ref{0911} (and the remark right after Lemma \ref{0911}), for
 any positive integers $k_1,\dots, k_m$, we have
$$(\Phi_{n_i}^{k_i})+ (\Phi_{n_s}^{k_s})=(1) \quad \text{ in }\Z[1/p][q].$$
By the Chinese remainder theorem, we have
$$\frac{ \Z[1/p][q]}{\left(\prod_{i=1}^m \Phi_{n_i}^{k_i}\right)}\;\simeq\;
 \prod_{i=1}^m
\frac{\Z[1/p][q]}{\left(\Phi_{n_i}^{k_i}\right)} .
$$
Taking the inverse limit, we get \eqref{0912}.
\end{proof}

Let $\pi_j: \calS_p \to \calS_{p,j}$ denote the projection onto the
$j$th component in the above decomposition.

\begin{lemma} Suppose $\xi$ is a root of unity of order $r= p^j r'$,
with $(r',p)=1$. Then
for any $x\in \calS_p$, one has
$$ \ev_\xi(x) = \ev_\xi(\pi_j(x)).$$
If $i\neq j$ then $\ev_\xi(\pi_i(x))=0$.

\label{1100}
\end{lemma}

\begin{proof} Note that $\ev_\xi(x)$ is the image of $x$ under the projection
$\calS_p \to \calS_p/(\Phi_r(q))= \Z[1/p][\xi]$.
It remains to notice that $\calS_{p,i}/(\Phi_r(q))=0$ if $i\neq j$.
\end{proof}


\subsection{Splitting of $\calS_b$}

Suppose $p$ is a prime divisor of $b$. Let
$$ \calS_{b}^{p,0} := \Z[1/b][q]^{\N_p} \qquad \text {and } \quad
\calS_{b}^{p,\bz} := \Z[1/b][q]^{p\N}\simeq \prod_{j>0}\Z[1/b][q]^{p^j\N_p}
$$

We have similarly
$$ \calS_b = \calS_{b}^{p,0} \times \calS_{b}^{p,\bz}$$
with canonical projections $\pi^p_0 : \calS_b \to \calS_b^{p,0}$ and
$\pi^p_\bz : \calS_b \to \calS_b^{p,\bz}$.
Note that if $b=p$, then $\calS_{p}^{p,0}=\calS_{p,0}$ and
$\calS_{p}^{p,\bz}=\prod_{j>0} \calS_{p,j}$.
As before we set $\calS_{b,0}:=\Z[1/b][q]^{\N_b}$ and
$\pi_0:\calS_b\to\calS_{b,0}$.

Suppose $f \in \calS_b$.
If $\xi$ is a root of unity of order coprime with $p$, then
$\ev_\xi(f) = \ev_\xi(\pi^p_0(f))$. Similarly, if the order of $\xi$
is divisible by $p$, then
$\ev_\xi(f) = \ev_\xi(\pi^p_\bz(f))$.

\subsection{Properties of the  ring $\cR_b$}
For any $b\in \N$, we have
\[
\R_b\simeq\Z[1/b][q]^{\N_2}
\]
since the sequence $(q;q^2)_k,\, k\in \N$, is cofinal to $\Phi^*_{\N_2}$.
Here $\N_2$ is the set of all odd numbers.

Let $\{p_i\,|\, i=1,\dots,m\}$ be the set of  all distinct  {\em odd}
prime divisors of $b$.
For $\bn=(n_1,\dots,n_m)$,   a tuple
of numbers $n_i\in \N$, let $\bp^{\bn}=\prod_{i}p_i^{n_i}$.
Let $S_{\bn}:=\bp^{\bn}\N_{2b}$. Then
$\N_2=\amalg_{\bn}\, S_\bn$.
Moreover, for
$a\in S_{\bn},$ $a'\in S_{\bn'}$, we have
$(\Phi_{a}(q),\Phi_{a'}(q))=(1)$ in $\Z[1/b]$ if $\bn\neq\bn'$.
In addition, each $S_\bn$ is $\Z[1/b]$--connected.
An argument similar to that for Equation \eqref{0912} gives
\[
\cR_b\simeq\prod_{\bn}\Z[1/b][q]^{S_\bn}.
\]
In particular, $\R^{p_i,0}_b:=\Z[1/b][q]^{\N_{2p_i}}$ and
$\R^{p_i,\bz}_b:=\Z[1/b][q]^{p_i\N_2}$ for any $1\leq i\leq m$.
If $2\mid b$, then  $\R^{2,0}_b$ coincides with $\R_b$.

Let $T$ be an infinite set  of powers of an odd prime not dividing $b$ and
let $P$ be an
infinite set of odd primes not dividing $b$.

\begin{proposition}\label{main-corProof} With the above notations,
one has the following.
\begin{itemize}
\item[(a)] For any   $l\in S_\bn$,
the Taylor map
$\TT_l : \Z[1/b][q]^{S_\bn}\;\to  \;\Z[1/b][e_{l}][[q-e_{l}]]$
is injective.


\item[(b)] Suppose $f, g\in \Z[1/b][q]^{S_{\bn}}$
such that $\ev_\xi(f)=\ev_\xi(g)$ for any root of unity
$\xi$ with $\text{ord}(\xi)\in \bp^\n T$, then $f=g$. The same holds true if $\bp^\n T$ is replaced by $\bp^\n P$.

\item[(c)] For odd $b$, the natural homomorphism $\rho_{\N,\N_2}:\calS_b\to \R_b$ is injective. If $2\mid b$,
then the natural homomorphism $\calS_{b}^{2,0}\to \R_{b}$ is an isomorphism.

\end{itemize}
\end{proposition}

\begin{proof}
(a) Since each $S_\bn$ is $\Z[1/b]$--connected in Habiro sense,
by \cite[Theorem 4.1]{Ha1}, for any $l\in S_\bn$
\be\label{inj}
\rho_{S, \{l\}}: \Z[1/b][q]^{S_\bn}\to \Z[1/b][q]^{\{ l\}}
\ee
is injective.  Hence $\TT_l = h \circ \rho_{S, \{l\}}$ is injective too.


(b) Since both sets contain infinitely many
numbers adjacent to $\bp^\n$, the claim follows from Theorem 6.1 in \cite{Ha1}.

(c) Note that for odd $b$
\[
\calS_b\simeq \prod_{\bn}\Z[1/b][q]^{S'_\bn}
\]
where $S'_{\bn}:=\bp^{\bn}\N_{b}$.
Further observe that $S'_{\bn}$ is $\Z[1/b]$--connected if $b$ is odd.
 Then by \cite[Theorem 4.1]{Ha1} the map
$$\Z[1/b][q]^{S'_\bn}\hookrightarrow \Z[1/b][q]^{S_{\bn}}$$
is an embedding. If $2\mid b$, then $\calS^{2,0}_b:=\Z[1/b][q]^{\N_2}\simeq\R_b$.
\end{proof}

Assuming Theorem \ref{main},
 Proposition \ref{main-corProof} (b)  implies
Proposition  \ref{main-cor}.


\section{On the  Ohtsuki series at roots of unity}\label{Oh-strategy}

The Ohtsuki series was defined
for $SO(3)$ invariants by Ohtsuki \cite{Oh}
and extended to all other Lie algebras by the third author \cite{Le2,Le3}.

In the works \cite{Oh,Le2,Le3},
it was proved that the sequence of quantum invariants at $e_{p}$, where $p$ runs through the set of primes, obeys some congruence properties that allow to define uniquely the coefficients of the Ohtsuki series. The proof of the existence of such congruence relations is difficult.
In \cite{Ha}, Habiro proved that Ohtsuki series
coincide with the Taylor expansion of the unified invariant at $q=1$
in the case of integral homology spheres; this result was generalized to
 rational homology spheres
by the third author \cite{Le}.

Here, we prove that the sequence of $SO(3)$ invariants at the $pr$th
roots $e_{r}e_p$, where $r$ is a
fixed odd number and $p$ runs through the set of primes,
obeys some congruence properties that allow to define uniquely the coefficients of the ``Ohtsuki series'' at $e_r$, which is coincident with the
Taylor expansion at $e_r$.

\subsection{Extension of $\Z[1/b][e_r]$} Fix an odd positive integer $r$.
Assume $p$ is a prime bigger than $b$ and $r$. The cyclotomic rings $\Z[1/b][e_{pr}]$ and $\Z[1/b][e_r]$ are extensions of $\Z[1/b]$ of degree $\varphi(rp)=\varphi(r) \varphi(p)$ and $\varphi(r)$, respectively.
Hence $\Z[1/b][e_{pr}]$ is an extension of $\Z[1/b][e_r]$ of degree $\varphi(p)=p-1$.
Actually, it is easy to see that for
\[
f_p(q):= \frac{q^p - e_r^p}{q-e_r},
\]
the map
\[
\frac{\Z[1/b,e_r][q]}{(f_p(q))} \;\to\; \Z[1/b][e_{pr}], \hspace{5mm}
q\mapsto e_p e_r,
\]
is an isomorphism.
We put $x=q-e_r$ and get
\begin{equation} \Z[1/b][e_{pr}]\simeq \frac{\Z[1/b,e_r][x]}{(f_p(x+e_r))}\; .
\label{5500}
\end{equation}
Note that
\[
f_p(x+ e_r) = \sum_{n=0}^{p-1} \binom{p}{n+1} x^n e_r^{p-n-1}
\]
is a monic polynomial in $x$ of degree $p-1$, and the coefficient of $x^n$ in $f_p(x+e_r)$ is divisible by $p$ if $n \le p-2$.

\subsection{Arithmetic expansion of $\tau'_M$} Suppose $M$ is a rational
homology 3--sphere with $|H_1(M,\Z)|=b$.
By Theorem \ref{main}, for any root of unity $\xi$ of order $pr$
$$\tau'_M(\xi)\in \Z[1/b][e_{pr}] \simeq \frac{\Z[1/b,e_r][x]}{(f_p(x+e_r))}\; .$$
Hence we can write
\be\label{Ohtsuki}
\tau'_M(e_{r}e_{p})= \sum_{n=0}^{p-2} a_{p,n} x^n
\ee
where  $a_{p,n}\in \Z[1/b,e_r]$.
The following proposition shows that the coefficients $a_{p,n}$
stabilize as $p\to \infty$.
\begin{proposition}\label{main-cor1} Suppose $M$ is a
rational homology 3--sphere with $|H_1(M,\Z)|=b$, and $r$ is an odd
positive integer.
For every non--negative integer $n$,
there exists a unique invariant $a_n= a_n(M) \in \Z[1/b,e_r]$ such that
for every prime $p > \max (b,r)$,
 we have
\begin{equation}
a_n\equiv a_{p,n} \pmod p\;\;\; \text{in $ \Z[1/b,e_r]$ for} \;\;\;
0\le n \le  p-2.
\label{5501}
\end{equation}
Moreover, the formal series $\sum_{n} a_n (q-e_r)^n$ is equal to the Taylor
 expansion of the unified invariant $I_M$ at $e_r$.
\end{proposition}

\begin{proof}
The uniqueness of $a_n$ follows from the easy fact that if $a \in \Z[1/b,e_r]$ is divisible
by infinitely many rational primes $p$, then $a=0$.

Assume Theorem \ref{main} holds. We define $a_n$ to be the coefficient of $(q-e_r)^n$ in
 the Taylor series of $I_M$ at $e_r$, and will show that
Equation \eqref{5501} holds true.

Recall that $x= q-e_r$.  The diagram
$$\begin{CD} \Z[\frac{1}{b}][q]^{\N_2} @>>> \Z[\frac{1}{b}, e_r][q]^{r\N_2} @>>> \Z[\frac{1}{b},e_r][[x]] \\
@VV q \to e_re_pV     @VV /(f_p(q))V   @VV /(f_p(x+e_r))V \\
\Z[\frac{1}{b}][e_{rp}] @> e_re_p \to q >> \frac{\Z[\frac{1}{b}, e_r][q]}{(f_p(q))} @>>>\frac{ \Z[\frac{1}{b},e_r][[x]]}{(f_p(x+e_r))}
\end{CD}$$
is commutative. Here the middle and the right vertical maps are the
quotient maps
 by the corresponding ideals. Note that $I_M$ belongs to  the upper left
corner ring, its
Taylor series is the image in the upper right corner ring, while
the evaluation
 \eqref{Ohtsuki} is in the lower middle ring. Using the commutativity
at the lower right corner ring, we see that

$$ \sum_{n=0}^{p-2} a_{p,n}x^n  = \sum_{n=0}^{\infty} a_{n}x^n  \pmod{f_p(x+e_r)}
\quad \text{in} \quad \Z[1/b,e_r][[x]].$$
Since the coefficients of $f_p(x+e_r)$ up to degree $p-2$ are divisible by
$p$, we get the congruence \eqref{5501}.
\end{proof}

\begin{remark}  Proposition \ref{main-cor1}, when $r=1$, was the main result
of Ohtsuki \cite{Oh}, which leads to the development of the theory of finite
type invariant and the LMO invariant.

 When $(r,b)=1$, then Taylor series at $e_r$ determines and is determined
by the Ohtsuki series. But when, say, $r$ is a divisor of $b$, a priori
the two Taylor series, one at $e_r$ and the other at $1$, are independent. We suspect
that the Taylor series at $e_r$, with $r\mid b$, corresponds to a new type of LMO invariant.

\end{remark}


\section{Frobenius maps} \label{map-frob}
The proof of Theorem \ref{0078}, and hence of the main theorem,
 uses the Laplace transform method.
The aim of this section is to show that the image of the
Laplace transform, defined in Section \ref{laplace},
belongs to $\R_b$, i.e. that
certain roots of $q$ exist in $\R_b$.

\subsection{On the module $\Z[q]/(\Phi^k_{n}(q))$} Since cyclotomic
completions are built from modules like $\Z[q]/(\Phi^k_{n}(q))$,
we first consider these modules.
Fix $n,k\ge 1$. Let
$$ E  := \frac{\Z[q]}{(\Phi^k_{n}(q))}, \quad \text { and } \quad G:=
\frac{\Z[e_{n}][x]}{(x^k)}\; .$$

The following is probably well--known.
\begin{proposition}\label{h_kInjektive}$\text{ }$
\begin{itemize}
\item[(a)] Both $E$ and $G$ are free $\Z$--modules of the same
 rank $k \varphi(n)$.

\item[(b)] The algebra map $h: \Z[q] \to \Z[e_n][x]$ defined by

$$ h (q) = e_n + x$$
descends to a well--defined algebra homomorphism, also denoted by $h$, from $E$ to $G$.
Moreover, the algebra homomorphism $h: E \to G$ is injective.

\end{itemize}
\end{proposition}

\begin{proof} (a) Since $\Phi^k_n(q)$ is a monic polynomial in $q$ of degree
$k \varphi(n)$, it is clear that
$$ E=  \Z[q]/(\Phi^k_n(q))$$
is a free $\Z$--module of rank $k\varphi(n)$.
Since $G = \Z[e_n]\otimes_\Z \Z[x]/(x^k)$, we see
 that $G$ is free over $\Z$ of rank $k\varphi(n)$.

(b) To prove that $h$ descends to a map $E \to G$,
 one needs to verify that
$h(\Phi^k_n(q))=0$.
Note that
$$ h(\Phi^k_n(q))= \Phi^k_{n}(x + e_n) =
\prod_{(j,n)=1} (x+e_n - e_n^j)^k.$$
When $j=1$, the factor is $x^k$, which is 0 in $\Z[e_n][x]/(x^k)$.
Hence $h(\Phi^k_n(q))=0$.

Now we prove that $h$ is injective. Let $f(q)\in\Z[q]$. Suppose $h(f(q))=0$,
or $f(x+e_n)=0$ in
$\Z[e_n][x]/(x^k)$.
It follows that $f(x+e_n)$ is divisible by $x^k$; or that $f(x)$ is divisible by
$(x-e_n)^k$.
Since $f$ is a polynomial with coefficients in $\Z$, it follows that $f(x)$ is
divisible by all Galois conjugates $(x-e_n^j)^k$ with $(j,n)=1$.
Then $f$ is
divisible by $\Phi^k_n(q)$. In other words, $f=0$ in
$E=  \Z[q]/(\Phi^k_n(q))$.
\end{proof}

\subsection{A Frobenius homomorphism} We use $E$ and $G$ of the
previous subsection. Suppose $b$ is a positive integer coprime with
$n$.
If $\xi$ is a primitive $n$th root of 1, i.e. $\Phi_n(\xi)=0$, then
$\xi^b$ is also a primitive $n$th root of $1$, i.e. $\Phi_n(\xi^b)=0$. It follows that
$\Phi_n(q^b)$ is divisible
by $\Phi_n(q)$.

Therefore
 the algebra map $F_b: \Z[q] \to \Z[q]$,
defined by $F_b(q)=q^b$,
descends to a well--defined algebra map, also denoted by $F_b$,
from $E$ to $E$. We want to understand the image $F_b(E)$.
\newcommand{\rk}{\operatorname{rk}}
\begin{proposition} The image $F_b(E)$ is a free $\Z$--submodule
of $E$ of maximal rank, i.e. $\rk(F_b(E)) = \rk(E)$. Moreover, the index
of $F_b(E)$ in $E$ is $b^{k(k-1)\varphi(n)/2}$.
\end{proposition}

\begin{proof}
 Using Proposition \ref{h_kInjektive} we identify $E$ with its image
$h(E)$ in $G$.

Let  $\tilde F_b: G \to G$ be the $\Z$--algebra homomorphism defined by
$\tilde F_b(e_n) =e_n^b, \tilde F_b(x)=
(x+e_n)^b - e_n^b$.

Note that $\tilde F_b(x) = b e_n^{b-1} x + O(x^2)$, hence $\tilde F_b(x^k)=0$.
It is easy to see that $\tilde F_b$ is
a well--defined algebra homomorphism, and that $\tilde F_b$ restricted to $E$ is exactly
$F_b$. Since $E$ is a lattice of
maximal rank in $G\otimes \Q$, it follows that the index of $F_b$ is exactly the
determinant of $\tilde F_b$, acting on
$G\otimes \Q$.

A basis of $G$ is $e_n^j x^l$, with $(j,n)=1, 0<j<n$ and $j=0$,
and $0\le l < k$. Note that
\[
\tilde F_b(e_n^j x^l) = b^le_n^{jb}  e_n^{(b-1)l} x^l + O(x^{l+1}).
\]
Since $(b,n)=1$, the set $e_n^{jb}$, with $(j,n)=1$ is the same as the set
$e_n^{j}$, with $(j,n)=1$.
Let $f_1: G\to G$ be the $\Z$--linear map defined by $f_1(e_n^{jb} x^l) =
e_n^{j} x^l$. Since $f_1$ permutes the basis elements,
its determinant  is $\pm 1$. Let $f_2: G\to G$ be the $\Z$--linear
map defined by $f_2(e_n^{j} x^l) = e_n^{j} (e_n^{1-b}x)^l$. The
determinant of $f_2$ is again $\pm 1$. This is because, for any fixed $l$,
$f_2$ restricts to the automorphism of $\Z[e_n]$ sending
$a$ to $ e^s_n a$, each of these maps has a well--defined inverse:
$a \mapsto e^{-s}_n a$.
Now
$$ f_1 f_2 \tilde F_b(e_n^j x^l) = b^l e_n^j x^l + O(x^{l+1})$$
can be described by an upper triangular matrix with $b^l$'s on the diagonal;
its determinant is equal to
$b^{k(k-1)\varphi(n)/2}$.
\end{proof}

From the proposition we see that if $b$ is invertible, then the index is equal to 1, hence we have

\begin{proposition} \label{onen}
For any $n$ coprime with $b$ and $k\in \N$, the Frobenius homomorphism $F_b:  \Z[1/b][q]/\left(\Phi^k_n(q)\right) \to \Z[1/b][q]/\left(\Phi^k_n(q)\right)$, defined by $F_b(q)= q^b$,
is an isomorphism.
\end{proposition}

\subsection{Frobenius endomorphism of $\calS_{b,0}$}
For finitely many $n_i\in\N_b$ and $k_i\in \N$, the Frobenius endomorphism

\[
F_b :\frac{\Z[1/b][q]}{\left(\prod_i\Phi^{k_i}_{n_i}(q)\right)}\to
\frac{\Z[1/b][q]}{\left(\prod_i\Phi^{k_i}_{n_i}(q)\right)}
\]
sending  $q$ to  $q^b$, is again well--defined.
Taking the inverse limit, we get an algebra endomorphism
\[
F_b: \Z[1/b][q]^{\N_b} \to \Z[1/b][q]^{\N_b}.
\]

\begin{Thm}\label{frob}
The Frobenius endomorphism
 $F_b: \Z[1/b][q]^{\N_b} \to \Z[1/b][q]^{\N_b}$, sending $q$ to $q^b$,
 is an isomorphism.
\end{Thm}

\begin{proof} For finitely many $n_i\in\N_b$ and $k_i\in \N$,
consider the natural algebra homomorphism
$$J:\frac{\Z[1/b][q]}{\left(\prod_i\Phi^{k_i}_{n_i}(q)\right)}
\to
\prod_i \frac{ \Z[1/b][q]}{\left(\Phi^{k_i}_{n_i}(q)\right)} .$$
This map is injective,
because  in the unique factorization domain $\Z[1/b][q]$, one has
\[
(\Phi_{n_1}(q)^{k_1} \dots \Phi_{n_s}(q)^{k_s})
= \bigcap_{j=1}^s \Phi_{n_j}(q)^{k_j} \, .
\]
Since the Frobenius homomorphism  commutes with $J$
and is an isomorphism on the target of $J$
by Proposition \ref{onen}, it is an isomorphism on the domain of $J$.
Taking the inverse limit, we get the claim.
\end{proof}

\subsection{Existence of $b$th root of $q$ in $\calS_{b,0}$}\label{qthroot}

\begin{lemma} Suppose $n$ and $b$ are coprime positive integers and $y \in \Q[e_n]$ such that $y^b=1$. Then $y =\pm 1$. If $b$ is odd then $y=1$.
\label{0923}
\end{lemma}

\begin{proof}Let $d\mid b$ be the order of $y$, i.e. $y$ is a primitive $d$th root of $1$. Then $\Q[e_n]$ contains $y$, and hence $e_d$.
Since $(n,d)=1$, one has $\Q[e_n]\cap \Q[e_d]=\Q$ (see e.g.
 \cite[Corollary of IV.3.2]{Lang}). Hence if $e_d \in \Q[e_n]$,
then $e_d\in\Q$, it follows that $d=1$ or $2$.
Thus $y=1$ or $y=-1$. If $b$ is odd,
 then $y$ cannot be $-1$.
\end{proof}

\begin{lemma} Let $b$ be a positive integer, $T\subset \N_b$, and
 $y \in \Q[q]^T$ satisfying $y^b=1$. Then $y =\pm 1$. If $b$ is odd then $y=1$.
\label{0913}
\end{lemma}
\begin{proof}  It suffices to show
that for any $n_1, n_2 \dots n_m \in T$, the ring
$\Q[q]/(\Phi^{k_1}_{n_1}\dots\Phi^{k_m}_{n_m})$ does not
contains a $b$th root of $1$ except possibly for $\pm 1$.
Using the Chinese remainder theorem, it  suffices to
consider the case where $m=1$.

The ring $\Q[q]/(\Phi_{n}^k(q))$ is isomorphic to $\Q[e_n][x]/(x^k)$,
by Proposition \ref{h_kInjektive}. If

$$y= \sum_{j=0}^{k-1} a_j x^j, \quad  a_j \in \Q[e_n]$$
satisfies  $y^b=1$,
then it follows that $a_0^b=1$.
By Lemma \ref{0923} we have have $a_0=\pm1$. One can easily see that
$a_1=\dots =a_{k-1}=0$. Thus $y=\pm 1$.
\end{proof}

In contrast with Lemma \ref{0913}, we have
\begin{proposition}
\label{cor9}
For any odd positive $b$,  and any subset $T\subset \N_b$,
the ring $\Z[1/b][q]^{T}$
contains a unique $b$th root of $q$, which is invertible in
$\Z[1/b][q]^T$.

For any even positive $b$,  and any subset $T\subset \N_b$,
the ring $\Z[1/b][q]^{T}$
contains two $b$th roots of $q$, which are invertible in
$\Z[1/b][q]^T$; one is the negative of the other.
\end{proposition}

\begin{proof}
Let us first consider the case $T=\N_b$. Since $F_b$ is an isomorphism
by Theorem \ref{frob}, we can define a $b$th root of $q$ by
\[
q^{1/b}:=F^{-1}_b (q)  \in \calS_{b,0}\,.
\]
If $y_1$ and $y_2$ are two $b$th root of the same element, then
their ratio $y_1/y_2$ is a $b$th root of 1.
From Lemma \ref{0913} it follows that  if $b$ is odd, there is only one
$b$th root of $q$ in $\Z[1/b][q]^{\N_b}$, and if $b$ is even,
there are 2 such roots, one is the minus of the other.
We will denote them $\pm q^{1/b}$.

Further it is known that $q$ is invertible in $\Z[q]^\N$ (see
\cite{Ha1}). Actually, there is an explicit expression
 $q^{-1}=\sum_n q^n (q;q)_n $.
Hence  $q^{-1}\in \Z[1/b][q]^{\N_b}$,
since the natural homomorphism from $\Z[q]^\N$ to $\Z[1/b][q]^{\N_b}$
 maps $q$ to $q$.  In a commutative ring,
if $x\mid y$ and $y$ is invertible, then so is $x$.
Hence any root of $q$ is invertible.

In the general case of $T \subset \N_b$, we use the natural map
$\Z[1/b][q]^{\N_b}\hookrightarrow \Z[1/b][q]^T$.
\end{proof}

\subsection*{ Relation with \cite{Le}}
By Proposition \ref{cor9},
  $\calS_{b,0}$
is isomorphic to the ring $\Lambda^{\N_b}_b:=\Z[1/b][q^{1/b}]^{\N_b}$
used in \cite{Le}. Furthermore, our invariant $\pi_0 I_M$ and the one
defined in \cite{Le} belong to $\calS_{b,0}$.
This follows from Theorem \ref{main} for $b$ odd, and
from Proposition \ref{main-corProof}(c) for $b$ even.
Finally, the invariant defined in \cite{Le} for $M$ divided
by the invariant of $\#_i L(b^{k_i}_i,1)$ (which is invertible in $\calS_{b,0}$
\cite[Subsection 4.1]{Le}) coincides with
$\pi_0 I_M$ up to factor $q^{\frac{1-b}{4}}$ by
Theorem \ref{main}, \cite[Theorem 3]{Le} and
Proposition \ref{main-corProof}(b).


\subsection{Another Frobenius homomorphism}
We define  another Frobenius type algebra homomorphism. The difference
of the two types of Frobenius homomorphisms is in the target spaces of
these homomorphisms.

Suppose $m$ is a positive integer.
Define the algebra homomorphism
$$G_m : R[q]^T \to R[q]^{mT} \quad \text{ by } \qquad G_m(q) = q^m.$$
Since $\Phi_{mr}(q)$ always divides $\Phi_{r}(q^m)$,  $G_m$ is well--defined.

\subsection{Realization of $q^{a^2/b}$ in $\calS_p$}  \label{defxb}
Throughout this subsection,
let $p$ be a prime or $1$.
Suppose
$b= \pm p^l$ for an $l\in \N$
and let $a$ be an integer.
Let $B_{p,j} = G_{p^j}(\calS_{p,0})$. Note that $B_{p,j}\subset \calS_{p,j}$.
If $b$ is odd, by Proposition \ref{cor9}  there is a unique $b$th root of $q$ in
 $\calS_{p,0}$; we denote it by $x_{b;0}$. If $b$ is even, by Proposition
 \ref{cor9} there are exactly two $b$th root of $q$, namely $\pm q^{1/b}$. We put
 $x_{b;0}=q^{1/b}$.
We define an element $z_{b,a} \in \calS_p$ as follows.

If $b\mid a$, let $z_{b,a} := q^{a^2/b} \in \calS_p$.

If $ b=\pm p^l \nmid a$, then $z_{b,a}\in \calS_p$ is
defined by specifying its projections
 $\pi_j(z_{b,a}):=z_{b,a;j}\in \calS_{p,j}$  as follows.
Suppose $a = p^s e$, with $(e,p)=1$. Then $s < l$. For $j >s$ let
 $z_{b,a;j}:=0$.
For $ 0\le j \le s$ let
\[
z_{b,a;j} := [G_{p^j}(x_{b;0})]^{a^2/p^j} = [G_{p^j}(x_{b;0})]^{e^2 \, p^{2s-j}}
\in B_{p,j} \subset \calS_{p,j}.
\]


Similarly, for $b=\pm p^l$ we define an element $x_b \in \calS_p$ as follows.
We put $\pi_0(x_b): =x_{b;0}$.
 For $j<l$, $\pi_j(x_b):= [G_{p^j}(x_{b;0})]^{p^j}$.
If $j\geq l$, $\pi_j(x_b):=q^{b}$. Notice that for $c=(b,p^j)$ we have
\[
\pi_j(x_b)=z_{b,c;j}.
\]


\begin{proposition} \label{eval_z}
Suppose $\xi$ is a root of unity of order $r = c r'$,
 where $c= (r,b)$. Then
$$ \ev_{\xi}(z_{b,a}) = \begin{cases} 0 & \text{ if } c \nmid a \\
 (\xi^c)^{a_1^2 b'_*} &  \text{ if } a=ca_1,
\end{cases}$$
where $b'_*$ is the unique element in $\Z/r'\Z$ such that
$  b'_\ast (b/c) \equiv 1 \pmod{r'}$.
Moreover,
\[
\ev_\xi(x_b)=(\xi^c)^{b'_\ast}\; .
\]

\end{proposition}

\begin{proof} Let us compute $\ev_\xi(z_{b,a})$. The case of
 $\ev_\xi(x_b)$ is completely analogous.

If $b\mid a$, then $c\mid a$, and the proof is obvious.

Suppose $b\nmid a$. Let $a=p^s e$ and $c= p^i$. Then $s <l$.
Recall that $z_{b,a} = \prod_{j=0}^\infty z_{b,a;j}$. By Lemma \ref{1100},
$$ \ev_\xi(z_{b,a}) = \ev_\xi(z_{b,a;i}).$$

If $c\nmid a$, then $i > s$. By definition, $z_{b,a;i}=0$, hence
the statement holds true.

It remains the case $c\mid a$, or $ i \le s$. Note that $\zeta = \xi^c$
is a primitive root of order $r'$ and  $(p, r')=1$.
Since $z_{b,a;i} \in B_{p,i}$,
$$ \ev_\xi(z_{b,a;i}) \in \Z[1/p][\zeta].$$
From the definition of $z_{b,a;i}$ it follows that $(z_{b,a;i})^{b/c} =
(q^c)^{a^2/c^2}$, hence after evaluation we have
\[
[\ev_\xi(z_{b,a;i})]^{b/c} =  (\zeta)^{a^2_1}.
\]
Note also that
$$ [(\xi^c)^{a^2_1 b'_*}]^{b/c} = (\zeta)^{a^2_1}.$$
Using Lemma \ref{0923} we conclude $\ev_\xi(z_{b,a;i}) =
(\xi^c)^{a^2_1 b'_*}$ if $b$ is odd, and
$\ev_\xi(z_{b,a;i}) =(\xi^c)^{a^2_1 b'_*}$ or
$\ev_\xi(z_{b,a;i}) =-(\xi^c)^{a^2_1 b'_*}$
if $b$ is even. Since $\ev_{1}(q^{1/b})=1$ and therefore
$\ev_{\xi}(q^{1/b})=\xi^{b_*}$ (and not $-\xi^{b*}$) we get the claim.
\end{proof}


\section{Invariant of lens spaces}\label{diag}
The purpose of this section is to prove Lemma
   \ref{0079}.
Throughout this section we will use the following notations.

Let $a$ and $b$ be coprime integers.
Choose $\hat{a}$ and $\hat{b}$ such that  $b\hat{b}+a\hat{a}=1$ with $0<\sn(a)\hat{a}<|b|$.
Notice that for $a=1$ we have $\hat{1}=1$ and $\hat{b}=0$.

Let $r$ be a fixed odd integer (the order of $\xi$).
For $l\in\Z$ coprime to $r$, let $l_*$ denote the inverse of $l$ modulo $r$. If $(b,r)=c$, let $b'_*$ denote the inverse of $b':=\frac{b}{c}$ modulo $r':=\frac{r}{c}$. Notice that for $c=1$, we have $b_*=b'_*$.

Further, we denote by  $\left(\frac{x}{y}\right)$
  the Jacobi symbol and by $s(a,b)$ the Dedekind sum
(see e.g. \cite{KM-Dedekind}).

\subsection{Invariants of lens spaces}
Let us compute the $SO(3)$ invariant of the lens space $M(b,a;d)$.
Recall that $M(b,a;d)$ is the lens space $L(b,a)$ together with a knot
$K$ inside colored by $d$ (see Figure 1).

\begin{proposition}\label{1409}
 Suppose $c=(b,r)$ divides  $d-\sn(a)\hat{a}$. Then
\begin{align*}
\tau'_{M(b,a;d)}(\xi)&= (-1)^{\frac{c+1}{2}\,\frac{\sn(ab)-1}{2}}
\left(\frac{|a|}{c}\right)
\left(\frac{1-\xi^{-\sn(a)db'_*}}{1-\xi^{-\sn(b)b'_*}}\right)^{\chi(c)}
\xi^{
4_*u-4_*b'_* \frac{a(\hat{a}-\sn(a)d)^2}{c}
}
\end{align*}
where
\[
u=12s(1,b)-12 \sn(b)s(a,b)+
\frac{1}{b}\left(a(1-d^2)+2(\sn(a)d-\sn(b))
+a(\hat{a}-\sn(a)d)^2\right) \in \Z
\]
and  $\chi(c)=1$ if $c=1$ and is zero otherwise.
If
$c\nmid (\hat{a}\pm d)$, $\tau_{M(b,a;d)}(\xi)=0\; .$
\label{lens}
\end{proposition}

In particular, it follows that $\tau_{L(b,a)}(\xi)=0$ if $c\nmid \hat{a}\pm 1$.

\begin{proof}
We consider first the case where
$b,a>0$. Since two lens spaces
$L(b,a_1)$ and $L(b,a_2)$ are homeomorphic if $a_1 \equiv a_2 \pmod{b}$, we can assume $a<b$.
Let $b/a$ be given by  a continued fraction
$$\frac{b}{a}=m_n-\frac{1}{\displaystyle m_{n-1}-\frac{1}{\displaystyle
    m_{n-2}-\dots \frac{1}{\displaystyle m_2-\frac{1}{\displaystyle
        m_1}}}}.
$$
Using the Lagrange identity
$$a-\frac{1}{b}=(a-1)+\frac{1}{\displaystyle 1+\frac{1}{\displaystyle
    (b-1)}}$$
we can  assume $m_i\geq 2$ for all $i$.
%

The $\tau_{M(b,a;d)}(\xi)$ can be computed in the same way as the invariant
$\xi_r(L(b,a),A)$ in \cite{LiLi}, after replacing
$A^2$ (respectively $A$) by $\xi^{2_*}$ (respectively $\xi^{4_*}$).
Representing the $b/a$--framed unknot in Figure \ref{figure:LensSpaceKnot}
by a Hopf chain (as
e.g.  in Lemma 3.1 of \cite{BL}), we have
\[
F_{L\sqcup K}(\xi,d)={\sum_{j_i}}^{\xi}\prod_{i=1}^{n}
q^{m_i  \frac{j_i^2-1}{4}}
\prod_{i=1}^{n-1} [j_i j_{i+1}]\cdot[j_n d][j_1]
=\frac{S_n(d)}{(\xi^{2_*}-\xi^{-2_*})^{n+1}}\cdot \xi^{-4_*\sum_{i=1}^{n}m_i}
\]
where
$$S_n(d)=
{\sum_{\substack{j_i=1 \\ \text{odd}}}^{2r}}
\xi^{\sum m_ij^2_i}
(\xi^{2_*j_1}-\xi^{-2_*j_1})
(\xi^{2_*j_1j_2}-\xi^{-2_*j_1j_2})\dots
(\xi^{2_*j_{n-1}j_n}-\xi^{-2_*j_{n-1}j_n})
(\xi^{2_*j_n d}-\xi^{-2_*j_n d})\, .$$


Using Lemmas  4.11, 4.12 and 4.20 of \cite{LiLi}\footnote{There are misprints
in Lemma 4.21:
$q^*\pm n$ should be replaced by $q^*\mp n$ for
$n=1,2$.}
(and replacing
 $e_r$ by $\xi^{4_*}$, $c_n$ by $c$, $N_{n,1}=p$ by $b$,
$N_{n-1,1}=q$  by $a$, $N_{n,2}=q^*$
 by $\hat{a}$ and $-N_{n-1,2}=p^*$ by $\hat{b}$), we get
\[
S_n(d)=(-2)^n (\sqrt{r} \epsilon(r))^n \sqrt{c} \epsilon(c)
\left(\frac{\frac{b}{c}}{\frac{r}{c}}\right)\left(\frac{a}{c}\right)
(-1)^{\frac{r-1}{2}\frac{c-1}{2}}
\cdot\sum_{\pm}\chi^{\pm}(d)
\xi^{-ca4_*b'_*
\left(\frac{d\mp \hat{a}}{c}\right)^2
\pm 2_*\hat{b}(d\mp \hat{a}) +4_*\hat{a} \hat{b}}
\]
where  $\chi^{\pm}(d)=\pm 1$ if
$c\mid d\mp \hat{a}$ and is zero otherwise.
Further $\epsilon(x)=1$ if
$x \equiv 1 \pmod 4$ and $ \epsilon(x)=I$ if $x \equiv 3 \pmod 4$.
This implies the second claim of the lemma.

Note that when $c=1$, both $\chi^\pm(d)$ are nonzero.
If $c>1$ and  $c \mid (d-\hat{a})$, $\chi^+(d)=1$, but $\chi^- (d)=0$.
 Indeed, for $c$ dividing $d-\hat{a}$,
$c \mid (d+\hat{a})$ if and only if $c\mid \hat{a}$ which is impossible, because
$c\mid b$ but $(b,\hat{a})=1$.


Inserting the last formula into the Definition \eqref{def_qi}
we get
\[
\tau_{M(b,a;d)}(\xi)=
\frac{S_n(d)}{\xi^{2_*}-\xi^{-2_*}}
\left(-2 \xi^{-3\cdot4_*} \sum_{j=1}^{r} \xi^{4_*j^2}\right)^{-n}
\xi^{-4_* \sum_{i=1}^{n}m_i}
\]
where we used that $\sigma_+=n$ and $\sigma_-=0$
(compare \cite[p. 243]{KM-Dedekind}).
From
$\sum_{j=1}^{r}\xi^{4_*j^2}=\epsilon(r)\sqrt{r}$,
we obtain
\[
\tau_{M(b,a;d)}(\xi)=(-1)^{\frac{(c-1)(r-1)}{4}}\epsilon(c)
\left(\frac{b'}{r'}\right)
\left(\frac{a}{c}\right)
\sqrt{c}
\; \frac{(1-\xi^{- db'_*})^{\chi(c)}}{\xi^{2_*}-\xi^{-2_*}}
\; \xi^{4_*(3n-\sum_i m_i)
-4_*\hat{b}(\hat{a}-2d)
-4_*b'_*\frac{a(d-\hat{a})^2}{c}}\,.
\]
Applying the following
formulas for the Dedekind sum (compare \cite[Theorem 1.12]{KM-Dedekind})
\be\label{Dedekind}
3n- \sum_i m_i=-12 s(a,b)+\frac{a+\hat{a}}{b}\;,
\;\;\;\;-3+b=12 s(1,b)-\frac{2}{b}
\ee
and dividing the formula for $\tau_{M(b,a;d)}(\xi)$ by
 the formula for $\tau_{L(b,1)}(\xi)$
we get
\[
\tau'_{M(b,a;d)}(\xi)=
\left(\frac{a}{c}\right)
\left(\frac{1-\xi^{-db'_*}}{1-\xi^{-b'_*}}\right)^{\chi(c)}
\xi^{4_*u-4_*b'_*\frac{a(d-\hat{a})^2}{c}}
\]
where
\[
u=
-12s(a,b)+12s(1,b)+
\frac{1}{b}\left(a+\hat{a}
-2
-\hat{b}b(\hat{a}-2d)\right).
\]
Notice, that $u\in\Z$.
Further observe, that by
using $a\hat{a}+b\hat{b}=1$, we get
\[
a+\hat{a}-2-\hat{b}b(\hat{a}-2d)=
2(d-1)+a(1-d^2)+a(\hat{a}-d)^2.
\]
This implies the result for  $0<a<b$.

To compute
$\tau_{M(-b,a;d)}(\xi)$, observe
 that $\tau_{M(b,-a;d)}=\tau_{M(-b,a;d)}$
is equal to the complex conjugate
of $\tau_{M(b,a;d)}$. The ratio
$$
\tau'_{M(-b,a;d)}(\xi)=\frac{\;\;\overline{\tau_{M(b,a;d)}(\xi)}\;\;}
{\tau_{L(b,1)}(\xi)}
$$
can be computed analogously. Using $\overline{\epsilon(c)}=(-1)^{\frac{c-1}{2}}\epsilon(c)$,
 we have for $a,b>0$
\[
\tau'_{M(-b,a,d)}(\xi) =(-1)^{\frac{c+1}{2}}
\left(\frac{a}{c}\right)
\left(\frac{1- \xi^{db'_*}}{1-\xi^{-b'_*}}\right)^{\chi(c)}
\xi^{4_*\tilde{u}+4_*b'_*\frac{a(d-\hat{a})^2}{c}}
\]
where
\[
\tilde{u}=12s(a,b)+12s(1,b)+\frac{1}{b}
\left(
-a-\hat{a}-2+\hat{b}b(\hat{a}-2d)
\right)
\]
Using $s(a,b)=s(a,-b)=-s(-a,b)$, we get the result.

\end{proof}

\begin{example*} For $b>0$, we have
\[
\tau'_{L(-b,1)}(\xi)
=
(-1)^{\frac{c+1}{2}-{\chi(c)}}\,
\xi^{2_*(b-3)+b_*\chi(c)}
\; .
\]
\end{example*}

\subsection{Proof of Lemma \ref{0079}}
Assume $b=\pm p^l$ and $p$ is prime.
We have to define the unified invariant of
  $M^\ve(b,a) := M(b,a;d(\ve))$,
where $d(0)=1$ and $d(\bz)$ is the smallest odd positive
integer such that $\sn(a)ad(\bz) \equiv 1 \pmod {b}$.
First observe  that such $d(\bz)$ always exists. Indeed,
if $p$ is odd, we can achieve this by adding $b$, otherwise
 the inverse of any odd number modulo $2^l$ is odd.

Recall that we denote the unique positive $b$th root of $q$
in $S_{p,0}$ by $q^{\frac{1}{b}}$.
We define the unified
invariant $I_{M^\ve(b,a)}\in \R_b$ as follows.
If $p\neq 2$, then $I_{M^\ve(b,a)}\in \calS_p$ is defined by specifying
its projections
\[
\pi _j I_{M^\ve(b,a)}
:=
\,
\begin{cases}
q^{3 s(1,b)-3\sn(b)\, s(a,b)}
&\text{if } j=0, \;\ve =0
\\&\\
(-1)^{\frac{p^j+1}{2}\,\frac{\sn(ab)-1}{2}}
\left(\frac{|a|}{p}\right)^j \,
q^{\frac{u'}{4}}
&\text{if } 0<j<l, \;\ve=\bz
\\&\\
(-1)^{\frac{p^l+1}{2}\,\frac{\sn(ab)-1}{2}}
\left(\frac{|a|}{p}\right)^l q^{\frac{u'}{4}}
&\text{if } j\geq l, \;\ve=\bz
\end{cases}
\]
where $u':=u-\frac{a(\hat{a}-\sn(a)d(\bz))^2}{b}$ and
 $u$ is defined in Proposition \ref{1409}.
 If $p=2$, then only  $\pi_0 I_{M(b,a)}\in \calS_{2,0}=\R_2$ is non--zero
and it is defined to be $q^{3s(1,b)-3 \sn(b)\; s(a,b)}$.

The $I_{M^{\varepsilon}(b,a)}$ is well--defined due to Lemma \ref{0622} below, i.e. all powers of $q$ in $I_{M^{\ve}(b,a)}$ are integers for $j>0$ or lie in $\frac{1}{b}\Z$ for $j=0$.
Further, for b odd (respectively even) $I_{M^\ve(b,a)}$ is invertible in
$\calS_{p}^{p,\varepsilon}$ (respectively $\R_p^{p,\varepsilon}$) since $q$ and $q^{\frac{1}{b}}$
are invertible in these rings.

In particular, for odd $b=p^l$, we have $I_{L(b,1)}=1$, and
\[
\pi _j I_{L(-b,1)}
=
\,
\begin{cases}
q^{\frac{b-3}{2}+\frac{1}{b}}
&\text{if } j=0
\\&\\
(-1)^{\frac{p^j+1}{2}}q^{\frac{b-3}{2}}
&\text{if } 0<j<l, \; p \text{ odd} \\&\\
(-1)^{\frac{p^l+1}{2}}q^{\frac{b-3}{2}}
&\text{if } j\geq l, \; p \text{ odd} \, .
\end{cases}
\]

It is left to show, that for any $\xi$ of order $r$ coprime with $p$, we have
\[
\ev_\xi ( I_{M^0(b,a)})=\tau'_{M^0(b,a)}(\xi)\,
\]
and if $r=p^j k$ with $j>0$, then
\[
\ev_\xi ( I_{M^{\bz}(b,a)})=\tau'_{M^\bz(b,a)}(\xi)\,.
\]
For $\ve=0$, this follows directly from Propositions \ref{eval_z} and
\ref{1409} with $c=d=1$.
For $\ve=\bz$, we have $c=(p^j,b)>1$ and
we get the claim by using
Proposition \ref{1409} and
\begin{equation} \label{0623}
\xi^{\frac{a(\hat{a}-\sn(a)d(\bz))^2}{b}}=
\xi^{c\,\frac{a(\hat{a}-\sn(a)d(\bz))^2}{bc}}=
\xi^{bb'_*\,\frac{a(\hat{a}-\sn(a)d(\bz))^2}{bc}}=
\xi^{b'_*\,\frac{a(\hat{a}-\sn(a)d(\bz))^2}{c}},
\end{equation}
where for the second equality we use $c\equiv bb'_*\pmod{r}$.
Notice that due to part (2) of Lemma \ref{0622} below,
$b$ and $c$ divide $\hat{a}-\sn(a)d(\bz)$ and therefore all powers of $\xi$
in (\ref{0623}) are integers.
\qed

The following Lemma is used in the  proof of Lemma \ref{0079}.
\begin{lemma} \label{0622}
We have
\begin{itemize}
\item[(a)] $ 3 s(1,b) - 3\sn(b)\, s(a,b) \in \frac{1}{b}\Z$,
\item[(b)] $b\mid \hat{a}-\sn(a)d(\bz)$ and therefore $u'\in\Z$, and
\item[(c)] $4\mid u'$ for $d=d(\bz)$.
\end{itemize}
\end{lemma}

\begin{proof}
The first claim follows from the formulas (\ref{Dedekind})
for the Dedekind sum.
The second claim follows from the fact that $(a,b)=1$ and
\[
a(\hat{a}-\sn(a)d)=1-\sn(a)ad-b\hat{b}\equiv 0\pmod{b},
\]
since $d$ is chosen such that $\sn(a)ad\equiv 1\pmod{b}$.
For the third claim, notice that for odd $d$ we have
\[
4\mid (1-d^2) \;\;\text{ and }\;\;
4\mid 2(\sn(a)d-\sn(b)).
\]
\end{proof}


\section{Laplace transform} \label{laplace}
This section is devoted to the proof of Theorem \ref{0078}
by using Andrew's identity.
Throughout this section, let $p$ be a prime or $p=1$, and $b= \pm p^l$ for an $l\in \N$.
\subsection{Definition}
The Laplace transform is a
 $\Z[q^{\pm 1}]$--linear map
defined by
\begin{eqnarray*}
\cL_{b}: \Z[z^{\pm 1}, q^{\pm 1}] &\to& \calS_p \\
z^a &\mapsto& z_{b,a}.
\end{eqnarray*}
In particular, we put $\cL_{b;j}:=\pi_j \circ \cL_{b}$ and have $\cL_{b;j}(z^a)=z_{b,a;j} \in \calS_{p,j}$.

Further, for any
  $f \in \BZ[z^{\pm 1},q^{\pm 1}]$ and $n\in \Z$, we define
\[
\hat f:=f|_{z=q^n}\in \Z[q^{\pm n}, q^{\pm 1}]\,.
\]

\begin{lemma}
Suppose  $f \in \BZ[z^{\pm 1},q^{\pm 1}]$. Then
for a root of unity $\xi$ of odd order $r$,
\[
{\sum_{n}}^\xi q^{b\frac{n^2-1}{4}}  \hat f = \gamma_{b}(\xi) \,
\ev_\xi(\cL_{-b}(f)).
\]
\label{1001}
\end{lemma}

\begin{proof}
It is sufficient to consider the case $f=z^a$.
Then, by the same arguments as in
the proof of \cite[Lemma 1.3]{BL},
we
have
\be\label{four}
{\sum_{n}}^\xi q^{b\frac{n^2-1}{4}}\,
q^{ na} = \begin{cases}
0
&\text{if $c\nmid  a$}\\
(\xi^c)^{-{a_1^2 b'_*}}\, \gamma_b(\xi)  & \text{if $a=ca_1$}.
 \end{cases}
\ee
The result follows now from Proposition \ref{eval_z}.


\end{proof}

\subsection{Proof of Theorem \ref{0078}}
Recall that
\[
A(n,k) = \frac{\prod^{k}_{i=0}
\left(q^{n}+q^{-n}-q^i -q^{-i}\right)}{(1-q) \, (q^{k+1};q)_{k+1}}.
\]
We have to show that there exists an element $Q_{b,k} \in \R_b $ such that
for every root of unity $\xi$ of odd order $r$ one has
\[
\frac{{\sum_n}^{\hspace{-1.8mm}\xi } \; q^{b\frac{n^2-1}{4}} A(n,k) }{F_{U^b}(\xi)}
= \ev_\xi (Q_{b,k}).
\]
Applying Lemma \ref{1001} to $F_{U^b}(\xi)={\sum\limits_ n}^\xi q^{b\frac{n^2-1}{4}}
[n]^2$, we get for $c=(b,r)$
\be\label{1122}
F_{U^b}(\xi)=2\gamma_b(\xi) \;\ev_{\xi}\left(\frac{(1-x_{-b})^{\chi(c)}}{(1-q^{-1})(1-q)}\right)
,\ee
where as usual, $\chi(c)=1$ if $c=1$ and is zero otherwise.
We will prove
that for an odd prime $p$ and
any number $j\geq 0$ there exists an element
$Q_k(q,x_b,j) \in \calS_{p,j}$ such that
\be\label{imp}
\frac{1}{(q^{k+1};q)_{k+1}}
\,\cL_{b;j}\left( \prod_{i=0}^k (z+z^{-1} - q^i -q^{-i})   \right)
=
2\, Q_k(q^{\sn(b)}, x_{b},j).
\ee
If $p=2$ we will prove the claim for $j=0$ only, since $\calS_{2,0}\simeq\cR_{2}$.
The case $p=\pm 1$ was done e.g. in \cite{BBL}.
Theorem \ref{0078} follows then from Lemma \ref{1001} and \eqref{1122}
where $Q_{b,k}$ is defined by its projections
\[
\pi_j Q_{b,k}:=\;
\frac{1-q^{-1}}{(1-x_{-b})^{\chi(p^j)}}\; Q_k(q^{-\sn(b)},x_{-b},j).
\]
We split the proof of (\ref{imp})
into two parts. In the first part we will show that
there exists an element $Q_{k}(q,x_b,j)$ such
 that  Equality \eqref{imp}
holds. In the second part we show that $Q_k(q,x_b,j)$
lies in $\calS_{p,j}$.

\subsection*{Part 1,  $b$ odd case}
Assume $b=\pm p^l$ with $p\not=2$.
We split the proof into several lemmas.

\begin{lemma}\label{S_{b;j}(k,q)}
For $x_{b;j}:=\pi_j(x_b)$ and $c=(b,p^j)$,
\[
\cL_{b;j}\left( \prod_{i=0}^k (z+z^{-1} - q^i -q^{-i})   \right)=
2\, (-1)^{k+1} \, \qbinom{2k+1}{k} \, S_{b;j}(k,q),
\]
where
\begin{equation}\label{unif-s}
S_{b;j}(k,q):=1+\sum_{n=1}^{\infty}\frac{q^{(k+1)cn}(q^{-k-1};q)_{cn}}
{(q^{k+2};q)_{cn}}
(1+q^{cn}) x_{b;j}^{n^2}.
\end{equation}
\end{lemma}

Observe that for  $n> \frac{k+1}{c}$ the term $(q^{-k-1};q)_{cn}$ is zero and therefore the sum in \eqref{unif-s} is finite.

\begin{proof}
Since $\cL_{b}$ is invariant under $z \to z^{-1}$ one has
\[
\cL_b\left(\prod_{i=0}^k (z+z^{-1} - q^i -q^{-i})\right) =
-2\cL_{b} (z^{-k}(zq^{-k};q)_{2k+1}),
\]
and the $q$--binomial theorem (e.g. see \cite{GR}, II.3) gives
\begin{equation}\label{qbinomial}
z^{-k}(zq^{-k};q)_{2k+1}=
(-1)^k \sum_{i=-k}^{k+1}(-1)^i\qbinom{2k+1}{k+i}z^i.
\end{equation}
Notice that $\cL_{b;j}(z^a)\not=0$ if and only if $c \mid a$. Applying $\cL_{b;j}$ to the RHS of (\ref{qbinomial}), only the terms with
$c \mid i$ survive and therefore
\[
\cL_{b;j}\left(z^{-k}(zq^{-k};q)_{2k+1}\right)=
(-1)^{k} \sum_{n=-\lfloor k/c \rfloor}^{\lfloor (k+1)/c \rfloor} (-1)^{cn}
\qbinom{2k+1}{k+cn} z_{b,cn;j}.
\]
Separating the case $n=0$ and combining positive and negative $n$ this is equal to
\[
(-1)^{k}\qbinom{2k+1}{k}
+(-1)^k \sum_{n=1}^{\lfloor (k+1)/c \rfloor} (-1)^{cn}\left(\qbinom{2k+1}{k+cn}+\qbinom{2k+1}{k-cn}\right) z_{b,cn;j},
\]
where we use the convention that $\qbinom{x}{-1}$ is put to be zero for positive $x$.
Further,
\[
\qbinom{2k+1}{k+cn} +\qbinom{2k+1}{k-cn} =\frac{\{k+1\}}{\{2k+2\}}\qbinom{2k+2}{k+cn+1}(q^{cn/2}+q^{-cn/2})
\]
and
\[
\frac{\{k+1\}}{\{2k+2\}}\qbinom{2k+2}{k+cn+1}\qbinom{2k+1}{k}^{-1}=
(-1)^{cn}q^{(k+1)cn+\frac{cn}{2}}\frac{(q^{-k-1};q)_{cn}}{(q^{k+2};q)_{cn}}
.
\]
Using $z_{b,cn;j}=(z_{b,c;j})^{n^2}=x_{b;j}^{n^2}$ we get the result.
\end{proof}

To define $Q_k(q,x_b,j)$ we will need
Andrew's identity (3.43) of \cite{And}:
\begin{eqnarray*}
&&\hspace{-6mm}\sum_{n\geq 0} (-1)^n\alpha_n t^{-\frac{n(n-1)}{2}+sn+Nn} \frac{(t^{-N})_n}{(t^{N+1})_n}
\prod_{i=1}^{s}\frac{(b_i)_n (c_i)_n}{b_i^nc_i^n(\frac{t}{b_i})_n(\frac{t}{c_i})_n}
=\\
&&
\hspace{-4mm}
\frac{(t)_N(\frac{q}{b_sc_s})_N}{(\frac{t}{b_s})_N(\frac{t}{c_s})_N}
\sum_{n_s\geq \cdots\geq n_2\geq n_1 \geq 0}
\beta_{n_1}
\frac{t^{n_s}(t^{-N})_{n_s}(b_s)_{n_s}(c_s)_{n_s}}{(t^{-N} b_s c_s)_{n_s}}
\prod_{i=1}^{s-1}
\frac{t^{n_i}}{b_i^{n_{i}} c_i^{n_{i}}}
\frac{(b_i)_{n_{i}}(c_i)_{n_{i}}}{(\frac{t}{b_i})_{n_{i+1}}(\frac{t}{c_i})_{n_{i+1}}}
\frac{(\frac{t}{b_i c_i})_{n_{i+1}-n_i}}{(t)_{n_{i+1}-n_i}}\;.
\end{eqnarray*}
Here and in what follows we use the notation $(a)_n:=(a;t)_n$ .
The special Bailey pair $(\alpha_n,\beta_n)$ is chosen as follows
\[
\begin{array}{rclcrcl}
\alpha_0&=&1, &&\alpha_n&=&(-1)^nt^{\frac{n(n-1)}{2}}(1+t^n)\\
\beta_0&=&1, &&\beta_n&=&0 \;\;\;\text{ for } n\geq 1.
\end{array}
\]

\begin{lemma}\label{LHS}
$S_{b;j}(k,q)$ is equal to the LHS of Andrew's identity
with the parameters fixed below.
\end{lemma}



\begin{proof}
Since
\[
 S_{b;j}(k,q)=S_{-b;j}(k,q^{-1}),
\]
it is enough to look at the case when $b>0$. Define $b':=\frac{b}{c}$ and
let $\omega$ be a $b'$th primitive root of unity. For simplicity,
put $N:=k+1$ and $t:=x_{b;j}$.
Using the following identities
\begin{eqnarray*}
(q^{y};q)_{cn}&=&\prod_{l=0}^{c-1}(q^{y+l};q^c)_n\\
(q^{yc};q^c)_n&=&\prod_{i=0}^{b'-1}(\omega^i t^y;t)_n,
\end{eqnarray*}
where the later is true due to $t^{b'}=x_{b;j}^{b'}=q^c$ for all $j$, and
choosing a $c$th root of $t$ denoted by $t^{\frac{1}{c}}$
we can see that
\[
S_{b;j}(k,q)=1+\sum_{n=1}^{\infty}
\prod_{i=0}^{b'-1}\prod_{l=0}^{c-1}
\frac{
(\omega^i t^{\frac{-N+l}{c}})_{n}
}{
(\omega^i t^{\frac{N+1+l}{c}})_{n}
}
(1+t^{b'n})t^{n^2+b'Nn}.
\]

Now we choose the parameters for Andrew's identity as follows.
We put
 $a:=\frac{c-1}{2}$, $d:=\frac{b'-1}{2}$ and $m:=\lfloor \frac{N}{c}\rfloor$.
For $l\in\{1,\ldots,c-1\}$ there exist unique $u_l,v_l\in\{0,\ldots,c-1\}$ such that
$u_l\equiv N+l\pmod{c}$ and $v_l\equiv N-l\pmod{c}$. Note that $v_l=u_{c-l}$. We define
$U_l:=\frac{-N+u_l}{c}$ and $V_l:=\frac{-N+v_l}{c}$. Then
$U_l, V_l \in \frac{1}{c}\Z$ but $U_l+V_l \in \Z$.
We define
\[
\begin{array}{rclrcll}
b_l&:=&t^{U_l}, &c_l&:=&t^{V_l} & \text{for }\; l=1,\ldots, a, \\
b_{a+i}&:=&\omega^it^{-m},
&c_{a+i}&:=&\omega^{-i}t^{-m}
& \text{for }\;i=1,\ldots,d, \\
b_{a+ld+i}&:=&\omega^it^{U_l},
&c_{a+ld+i}&:=&\omega^{-i}t^{V_l}
& \text{for }\;i=1,\ldots,d \text{ and } l=1,\ldots, c-1, \\
b_{g+i}&:=&-\omega^{i}t, &c_{g+i}&:=&-\omega^{-i}t &\text{for }\;i=1,\ldots,d,\\
b_{s-1}&:=&t^{-m}, &c_{s-1}&:=&t^{N+1}, &\\
b_s&\rightarrow&\infty,&c_s&\rightarrow&\infty.&\\
\end{array}
\]
where $g=a+cd$ and $s=(c+1)\frac{b'}{2}+1$.

We now calculate the LHS of Andrew's identity.
Using the notation
\[
(\omega^{\pm 1} t^x)_n=(\omega t^x)_n(\omega^{-1} t^x)_n
\]
and the identities
\[
\lim_{c\to \infty} \frac{(c)_n}{c^n}=(-1)^n t^{\frac{n(n-1)}{2}} \;\;\text{ and }\;\; \lim_{c\to\infty} \left(\frac{t}{c}\right)_n=1
\]
we get
\begin{eqnarray*}
LHS&=&1+
\sum_{n\geq 1}
t^{n(n-1+s+N-y)}\;
(1+t^n)\frac{(t^{-N})_n}{(t^{N+1})_n}
\cdot
\prod_{l=1}^{a}
	\frac{
		(t^{U_l})_n (t^{V_l})_n
	}{
		(t^{1-U_l})_n (t^{1-V_l})_n
	}	
\cdot
\prod_{i=1}^{d}
	\frac{
		(\omega^{\pm i} t^{-m})_n
	}{
		(\omega^{\pm i} t^{1+m})_n
	}
\\	&& \hspace{3cm}
\cdot\prod_{i=1}^{d}\prod_{l=1}^{c-1}
	\frac{
		(\omega^i t^{U_l})_n
		(\omega^{-i}t^{V_l})_n
	}{
	(\omega^{-i} t^{1-U_l})_n
	(\omega^{i}t^{1-V_l})_n	
	}
\cdot
\prod_{i=1}^{d}
	\frac{
		(-\omega^{\pm i}t)_n
	}{
		(-\omega^{\pm i})_n
	}
\cdot
\frac{
	(t^{-m})_n (t^{N+1})_n
}{
	(t^{1+m})_n (t^{-N})_n
}
\end{eqnarray*}
where
 \[
 y:=
\sum_{l=1}^{a} (U_l+V_l)+\sum_{i=1}^{d}\sum_{l=1}^{c-1}(U_l+V_l)-m(2d+1)+2d+1+N.
\]
Since
$\sum_{l=1}^{c-1}(U_l+V_l)=2\sum_{l=1}^{a}(U_l+V_l)=2(-N+m+\frac{c-1}{2})$ and
$2d+1=b'$, we have
\[
n-1+s+N-y=n+Nb'.
\]
Further,
\[
\prod_{i=1}^{d}	\frac{(-\omega^{\pm i}t)_n}{(-\omega^{\pm i})_n}
=
\prod_{i=1}^{b'-1}\frac{1+\omega^i t^n}{1+\omega^i}=\frac{1+t^{b'n}}{1+t^n}	
\]
and
\begin{eqnarray*}
&&
\prod_{l=1}^{a}
	\frac{
		(t^{U_l})_n (t^{V_l})_n
	}{
		(t^{1-U_l})_n (t^{1-V_l})_n
	}	
\cdot
\prod_{i=1}^{d}
	\frac{
		(\omega^{\pm i} t^{-m})_n
	}{
		(\omega^{\pm i} t^{1+m})_n
	}
\cdot
\prod_{i=1}^{d}\prod_{l=1}^{c-1}
	\frac{
		(\omega^i t^{U_l})_n
		(\omega^{-i}t^{V_l})_n
	}{
	(\omega^{-i} t^{1-U_l})_n
	(\omega^{i}t^{1-V_l})_n	
	}
\cdot
\frac{
	(t^{-m})_n
}{
	(t^{1+m})_n
}
\\
&&\hspace{105mm}=
\prod_{i=0}^{b'-1}\prod_{l=0}^{c-1}
\frac{
(\omega^i t^{\frac{-N+l}{c}})_{n}
}{
(\omega^i t^{\frac{N+1+l}{c}})_{n}
}.
\end{eqnarray*}
Taking all the results together, we see that the LHS is equal to
$S_{b;j}(k,q)$.

\end{proof}

Let us now calculate the RHS of Andrew's identity with parameters chosen as above.
For simplicity, we put $\delta_j:=n_{j+1}-n_j$.
Then the RHS is given by
\begin{eqnarray*}
RHS
&=&
(t)_N\sum_{n_s\geq \cdots\geq n_2 \geq n_1=0}
\frac{
t^{x}\cdot (t^{-N})_{n_s}(b_s)_{n_s}(c_s)_{n_s}
}{
\prod_{i=1}^{s-1} (t)_{\delta_{i}}(t^{-N}b_sc_s)_{n_s}
}
\cdot
\frac{
(t^{-m})_{n_{s-1}}(t^{N+1})_{n_{s-1}}
(t^{m-N})_{\delta_{s-1}}
}{
(t^{m+1})_{n_s}(t^{-N})_{n_s}
}
\\
&&\hspace{7mm}
\cdot
\prod_{l=1}^{a}
\frac{
(t^{U_l})_{n_l}(t^{V_l})_{n_l}
(t^{1-U_l-V_l})_{\delta_{l}}
}{
(t^{1-U_l})_{n_{l+1}}(t^{1-V_l})_{n_{l+1}}
}
\cdot
\prod_{i=1}^{d}
\frac{
(\omega^{\pm i}t^{-m})_{n_{a+i}}
(t^{2m+1})_{\delta_{a+i}}
}{
(\omega^{\pm i}t^{m+1})_{n_{a+i+1}}
}
\frac{
(-\omega^{\pm i}t)_{n_{g+i}}
(t^{-1})_{\delta_{g+i}}
}{
(-\omega^{\pm i})_{n_{g+i+1}}
}
\\
&&\hspace{7mm}
\cdot
\prod_{i=1}^{d}\prod_{l=1}^{c-1}
\frac{
(\omega^it^{U_l})_{n_{a+ld+i}}
(\omega^{-i}t^{V_l})_{n_{a+ld+i}}
(t^{1-U_l-V_l})_{\delta_{a+ld+i}}
}{
(\omega^{-i}t^{1-U_l})_{n_{a+ld+i+1}}
(\omega^{i}t^{1-V_l})_{n_{a+ld+i+1}}
}
\end{eqnarray*}
where
\begin{eqnarray*}
x
&=&\sum_{l=1}^{a}(1-U_l-V_l)\,n_l
+\sum_{i=1}^{d}(2m+1)\,n_{a+i}
\\&&\hspace{2cm}
+\sum_{i=1}^{d}\sum_{l=1}^{c-1}(1-U_l-V_l)\, n_{a+ld+i}
-\sum_{i=1}^{d}n_{g+i}+(m-N)\,n_{s-1}+n_s.
\end{eqnarray*}
For $c=1$ or $d=0$,
we use the convention that empty products  are set to be 1 and empty sums are set to be zero.

Let us now have a closer look at the RHS.
Notice, that
\[
\lim_{b_s,c_s\to\infty}\frac{
(b_s)_{n_s}(c_s)_{n_s}
}{
(t^{-N}b_sc_s)_{n_s}
}
=
(-1)^{n_s}t^{\frac{n_s(n_s-1)}{2}}t^{Nn_s}.
\]
The term $(t^{-1})_{\delta_{g+i}}$ is zero unless
 $\delta_{g+i}\in \{0,1\}$.
Therefore, we get
\[
\prod_{i=1}^{d}\frac{
(-\omega^{\pm i}t)_{n_{g+i}}
}{
(-\omega^{\pm i})_{n_{g+i+1}}
}
=
\prod_{i=1}^{d}
(1+\omega^{\pm i}t^{n_{g+i}})^{1-\delta_{g+i}}.
\]
Due to the term $(t^{-m})_{n_s}$, we have
$n_s\leq m$ and therefore $n_i\leq m$ for all $i$.
Multiplying the numerator and denominator of each term of the RHS by
\begin{eqnarray*}
&&
\prod_{l=1}^{a} (t^{1-U_l+n_{l+1}})_{m-n_{l+1}}(t^{1-V_l+n_{l+1}})_{m-n_{l+1}}
\prod_{i=1}^{d} (\omega^{\pm i}t^{m+1+n_{a+i+1}})_{m-n_{a+i+1}}
\\
&&\hspace{35mm}\cdot\prod_{i=1}^{d}\prod_{l=1}^{c-1}
(\omega^{-i}t^{1-U_l+n_{a+ld+i+1}})_{m-n_{a+ld+i+1}}
(\omega^{i}t^{1-V_l+n_{a+ld+i+1}})_{m-n_{a+ld+i+1}}
\end{eqnarray*}
gives in the denominator
$\prod_{i=0}^{b'-1}\prod_{l=1}^{c-1}(\omega^it^{1-U_l})_m
\cdot
\prod_{i=1}^{b'-1}(\omega^it^{m+1})_m$.
%
%
This is equal to
\[
\prod_{l=1}^{c-1}(t^{b'(1-U_l)};t^{b'})_m
\cdot \frac{(t^{b'(m+1)};t^{b'})_m}{(t^{m+1};t)_m}
=
\frac{(q^{N+1};q)_{cm}}{(t^{m+1};t)_m}.
\]
Further,
\[
(t)_N(t^{N+1})_{n_{s-1}}=(t)_{N+n_{s-1}}=(t)_{m}(t^{m+1})_{N-m+n_{s-1}}.
\]
The term $(t^{-N+m})_{\delta_{s-1}}$ is zero unless $\delta_{s-1}\leq N-m$
and therefore
\[
\frac{(t^{m+1})_{N-m+n_{s-1}}}{(t^{m+1})_{n_s}}
=
(t^{m+1+n_s})_{N-m-\delta_{s-1}}.
\]
Taking the above calculations into account,
we get
\begin{equation}\label{RHS}
RHS=
\frac{(t;t)_{2m}}{(q^{N+1};q)_{cm}}
\cdot T_k(q,t)
\end{equation}
where
\begin{eqnarray*}
T_k(q,t)&:=&
\sum_{n_s\geq \cdots\geq n_2 \geq n_1=0}
(-1)^{n_s}t^{x'}\cdot
(t^{-m})_{n_{s-1}}\cdot
(t^{m+1+n_s})_{N-m-\delta_{s-1}}
\cdot\frac{
(t^{-N+m})_{\delta_{s-1}}
}{
\prod_{i=1}^{s-1} (t)_{\delta_i}
}
\\
&&\hspace{3mm}
\cdot
\prod_{l=1}^{a}(t^{1-U_l-V_l})_{\delta_l}
\cdot
\prod_{i=1}^{d} (t^{2m+1})_{\delta_{a+i}}(t^{-1})_{\delta_{g+i}}
\cdot
\prod_{i=1}^{d}\prod_{l=1}^{c-1}(t^{1-U_l-V_l})_{\delta_{a+ld+i}}
\\
&&\hspace{3mm}
\cdot\prod_{l=1}^{a}
(t^{U_l})_{n_l}(t^{V_l})_{n_l}
(t^{1-U_l+n_{l+1}})_{m-n_{l+1}}
(t^{1-V_l+n_{l+1}})_{m-n_{l+1}}
\cdot
\prod_{i=1}^{d}
(1+\omega^{\pm i}t^{n_{g+i}})^{1-\delta_{g+i}}
\\
&&\hspace{3mm}
\cdot
\prod_{i=1}^{d}
(\omega^{\pm i}t^{-m})_{n_{a+i}}
(\omega^{\pm i}t^{m+1+n_{a+i+1}})_{m-n_{a+i+1}}
\cdot
\prod_{i=1}^{d}
\prod_{l=1}^{c-1}
(\omega^{i}t^{U_l})_{n_{a+ld+i}}
(\omega^{-i}t^{V_l})_{n_{a+ld+i}}
\\
&&\hspace{3mm}
\cdot
\prod_{i=1}^{d}
\prod_{l=1}^{c-1}
(\omega^{-i}t^{1-U_l+n_{a+ld+i+1}})_{m-n_{a+ld+i+1}}
(\omega^{i}t^{1-V_l+n_{a+ld+i+1}})_{m-n_{a+ld+i+1}}
\end{eqnarray*}
and
$x':=x+\frac{n_s(n_s-1)}{2}+Nn_s$.

We  now define the element $Q_{k}(q,x_b,j)$   by
\[
Q_{k}(q,x_b,j):=
\left((-1)^{k+1}q^{-\frac{k(k+1)}{2}}\right)
^{\frac{1+\sn(b)}{2}}
\left(q^{(k+1)^2}\right)^{\frac{1-\sn(b)}{2}}
\frac{(x_{b;j};x_{b;j})_{2m}}{(q;q)_{N+cm}}
\;T_k(q,x_{b;j}).
\]

By Lemmas \ref{S_{b;j}(k,q)} and \ref{LHS}, Equation (\ref{RHS}) and the following Lemma \ref{bPosbNeg},
we  see that this element
satisfies Equation (\ref{imp}).

\begin{lemma}\label{bPosbNeg}
The following formula holds.
\[
(-1)^{k+1}\qbinom{2k+1}{k}(q^{k+1};q)_{k+1}^{-1}=
(-1)^{k+1}\frac{q^{-k(k+1)/2}}{(q;q)_{k+1}}
=
\frac{q^{-(k+1)^2}}{(q^{-1};q^{-1})_{k+1}}
\]
\end{lemma}

\begin{proof}
This is an easy calculation using
\[
(q^{k+1};q)_{k+1}=(-1)^{k+1} q^{(3k^2+5k+2)/4}\frac{\{2k+1\}!}{\{k\}!}.
\]
\end{proof}


\subsection*{Part 1, $b$ even case.}
Let  $b=\pm 2^l$. We have to prove Equality (\ref{imp}) only for $j=0$, i.e.
we have to show
\[
\frac{1}{(q^{k+1};q)_{k+1}}
\,\cL_{b;0}\left( \prod_{i=0}^k (z+z^{-1} - q^i -q^{-i})   \right)
=
2\, Q_k(q^{\sn(b)}, x_{b},0).
\]
The calculation works similar  to the odd case. Note that we have $c=1$
here.
This case was already done in \cite{BL} and \cite{Le}.
Since their approaches are slightly different and
for the sake of completeness, we will give the parameters
for Andrew's identity and the formula for $Q_{k}(q,x_{b},0)$ nevertheless.

We put
$t:=x_{b;0}$, $d:=\frac{b}{2}-1$, $\omega$ a $b$th root of unity
and choose a primitive square root $\nu$ of $\omega$. Define the parameters of Andrew's identity by
\[
\begin{array}{rclrcll}
b_{i}&:=&\omega^it^{-N},
&c_{i}&:=&\omega^{-i}t^{-N}
&  \;\text{for }\;i=1,\ldots,d, \\
b_{d+i}&:=&-\nu^{2i-1}t, &c_{d+i}&:=&-\nu^{-(2i-1)}t &\;\text{for }\;i=1,\ldots,d+1,\\
b_{b}&:=&-t^{-N},
&c_{b}&:=&-t^0=-1, &\\
b_{s-1}&:=&t^{-N}, &c_{s-1}&:=&t^{N+1}, &\\
b_s&\rightarrow&\infty,&c_s&\rightarrow&\infty,&
\end{array}
\]
where $s=b+2$.
Now we can define the element
\[
Q_{k}(q,x_b,0):=
\left((-1)^{k+1}q^{-\frac{k(k+1)}{2}}\right)
^{\frac{1+\sn(b)}{2}}
\left(q^{(k+1)^2}\right)^{\frac{1-\sn(b)}{2}}
\frac{(x_{b;0};x_{b;0})_{2N}}{(q;q)_{2N}}
\frac{1}{(-x_{b;0};x_{b;0})_N} T_k(q,x_{b;0})
\]
where
\begin{eqnarray*}
T_k(q,t)
&:=&
\sum_{n_{s-1}\geq \cdots \geq n_1=0}
(-1)^{n_{s-1}}t^{x''}
\cdot
\frac{
\prod_{i=1}^{d}
(t^{2N+1})_{\delta_{i}}
\cdot\prod_{i=1}^{d+1}(t^{-1})_{\delta_{d+i}}
\cdot
(t^{N+1})_{\delta_b}
}{
\prod_{i=1}^{s-2}(t)_{\delta_i}
}
\\
&&\hspace{13mm}
\cdot
(t^{-N})_{n_{s-1}}
\cdot (-t^{N+1+n_{s-1}})_{N-n_{s-1}}
\cdot(-t^{-N})_{n_{b}}
\cdot (-t)_{n_{b}-1}
\cdot (-t^{n_{s-1}+1})_{N-n_{s-1}}
\\
&&\hspace{13mm}
\cdot
\prod_{i=1}^{d}
(\omega^{\pm i}t^{-N})_{n_{i}}
(\omega^{\pm i}t^{N+1+n_{i+1}})_{N-n_{i+1}}
\cdot
\prod_{i=1}^{d+1}(1+\nu^{\pm(2i-1)}t^{n_{d+i}})^{1-\delta_{d+i}}
\end{eqnarray*}
and
$
x'':=\sum_{i=1}^{d}(2N+1)n_i-\sum_{i=1}^{d+1}n_{d+i}+\frac{n_{s-1}(n_{s-1}-1)}{2}+(N+1)(n_b+n_{s-1})$.
We use the notation $(a;b)_{-1}=\frac{1}{1-ab^{-1}}$.


\subsection*{Part 2}
We have to show  that $Q_{k}(q,x_b,j) \in\calS_{p,j}$,
where $j \in \N\cup\{0\}$ if $p$ is odd, and $j=0$ for $p=2$.
The following two lemmas do the proof.

\begin{lemma}
For $t=x_{b;j}$,
\[
T_k(q,t)\in \Z[q^{\pm 1}, t^{\pm1}].
\]
\end{lemma}

\begin{proof}
Let us first look at the case $b$ odd and positive.
Since for $a\not=0$, $(t^a)_n$ is always divisible by $(t)_n$,
it is easy to see that the denominator of each term of $T_k(q,t)$
divides its numerator.
Therefore we proved that $T_k(q,t)\in\Z[t^{\pm 1/c},\omega]$.
Since
\be\label{Sbj}
S_{b;j}(k,q)=
\frac{(t;t)_{2m}}{(q^{N+1};q)_{cm}}
\cdot T_k(q,t),
\ee
there are $f_0,g_0\in\Z[q^{\pm 1},t^{\pm 1}]$
such that $T_k(q,t)= \frac{f_0}{g_0}$. This implies that
$T_k(q,t)\in \Z[q^{\pm 1}, t^{\pm 1}]$
since $f_0$ and $g_0$ do not depend on $\omega$ and
the $c$th root of $t$.

The proofs for the even and the negative case work similar.

\end{proof}

\begin{lemma}
For $t=x_{b;j}$,
\[
\frac{(t;t)_{2m}}{(q;q)_{N+cm}}
\frac{1}{((-t;t)_{N})^{\lambda}}
\in\calS_{p,j}
\]
where $\lambda=1$ and $j=0$ if $p=2$,  and $\lambda=0$ and
$j\in \N\cup \{0\}$ otherwise.
\end{lemma}

\begin{proof}
Notice that
\[
(q;q)_{N+cm}=\widetilde{(q;q)}_{N+cm}(q^c;q^c)_{2m},
\]
where we use the notation
\[
\widetilde{(q^a;q)}_{n}:=
\prod_{\substack{j=0\\c\nmid (a+j)}}^{n-1}(1-q^{a+j}).
\]
We have to show
that
\[
\frac{(q^c;q^c)_{2m}}{(t;t)_{2m}}
\cdot
\widetilde{(q;q)}_{N+cm}\cdot
((-t;t)_{N})^{\lambda}
\]
is invertible in
$\Z[1/p][q]$ modulo
any ideal  $(f)=(\prod_{n} \Phi^{k_n}_n (q))$ where $n$ runs
through a subset of $p^j\N_p$.
Recall that in a commutative ring $A$, an element $a$ is invertible
in $A/(d)$ if and only if $(a)+(d)=(1)$. If $(a)+(d)=(1)$ and $(a)+(e)=(1)$,
multiplying together we get $(a)+(de)=(1)$. Hence, it is enough to
consider  $f=\Phi_{p^j n}(q)$ with $(n,p)=1$.
For any $X\in\N$, we have
\begin{eqnarray}\label{denomin1}
{\widetilde{(q;q)}_X}&=&\prod_{\text{\scriptsize{$\begin{array}{c}i=1 \\ c\nmid i \end{array}$}}}^{X} \prod_{d\mid i} \Phi_d(q),
\\\label{denomin2}
(-t;t)_X&=&\frac{(t^2;t^2)_X}{(t;t)_X}=\prod_{i=1}^{X}\prod_{d\mid i}\Phi_{2d}(t)
\\\label{denomin3}
\frac{(q^c;q^c)_X}{(t;t)_X}
&=&\frac{(t^{b'};t^{b'})_X}{(t;t)_X}=
\frac{
\prod_{i=1}^{X}\prod_{d\mid ib'}\Phi_d(t)
}{
\prod_{i=1}^{X}\prod_{d\mid i}\Phi_d(t)
}
\end{eqnarray}
for $b'=b/c$.
Recall that $(\Phi_r(q),\Phi_a(q))=(1)$ in $\Z[1/p][q]$ if either
$r/a$ is not a power of a prime or  a power
of $p$.
For $r=p^j n$ odd and $a$ such that $c\nmid a$, one
of the conditions is always satisfied.
Hence  \eqref{denomin1} is invertible in $\calS_{p,j}$.
 If $b=c$ or $b'=1$, \eqref{denomin2} and 
\eqref{denomin3}  do not contribute.
For $c<b$, notice that
$q$ is a $cn$th primitive root of unity in
$\Z[1/p][q]/(\Phi_{cn}(q)) = \Z[1/p][e_{cn}]$.
Therefore
$t^{b'}=q^{c}$ is an $n$th primitive root of unity.
Since $(n,b')=1$, $t$ must be a
primitive $n$th root of unity in $\Z[1/p][e_{cn}]$, too, and hence
$\Phi_n(t) = 0$ in that ring.
Since for $j$ with 
$(j, p) > 1$, $(\Phi_j(t),\Phi_n(t)) = (1)$
in $\Z[1/p][t]$, we have $\Phi_j(t)$ is invertible
in $\Z[1/p][e_{cn}]$, and therefore \eqref{denomin2} and 
\eqref{denomin3} are invertible, too.
\end{proof}


\appendix

\numberwithin{equation}{section}
\numberwithin{figure}{section}

\section{Proof of
Theorem \ref{GeneralizedHabiro}}
The appendix is devoted to the proof of Theorem \ref{GeneralizedHabiro},
a generalization of the  deep integrality
result of Habiro, namely Theorem 8.2 of \cite{Ha}. The existence
of this generalization and some ideas of the proof were kindly
communicated to us by Habiro.

\subsection*{Reduction to a result on values of the colored Jones polynomial}
We will use the notations
of \cite{Ha}. We put  $q=e^{h}$,
and $v=e^{h/2}$, where $h$ is a free parameter.
The quantum algebra $U_h=U_h(sl_2)$, generated
by $E, F$ and $H$, subject to some relations, is the
quantum deformation of the universal enveloping algebra
$U(sl_2)$.

Let $V_n$ be the unique $(n+1)$--dimensional
irreducible $U_h$--module.
In \cite{Ha},
Habiro defined a new basis
$\tilde{P}'_k$, $k=0,1,2,\dots$,  for the Grothendieck ring of
finite--dimensional
$U_h(sl_2)$--modules with
\[
\tilde{P}_k':=\frac{v^{\frac{1}{2}k(1-k)}}{\{k\}!} \, \prod_{i=0}^{k-1}
(V_1-v^{2i+1}-v^{-2i-1}).
\]

Put $\tilde{P}'_{\bk}=\{\tilde{P}'_{k_1},\ldots,\tilde{P}'_{k_m}\}$. It follows from Lemma 6.1 of \cite{Ha}  that we will have identity \eqref{Jones} of Theorem \ref{GeneralizedHabiro} if we put

\[
C_{L\sqcup L'}(\bk,\bj)=
J_{L\sqcup L'}\,(\tilde{P}'_{\bk},\bj)\,
\prod_{i}(-1)^{k_i}q^{k^2_i+k_i+1}\, .
\]

Hence
to prove Theorem \ref{GeneralizedHabiro} it is enough to show the following.

\begin{AThm}\label{main-integrality}
Suppose $L\sqcup L'$ is a colored framed link in $S^3$ such that
 $L$ has zero linking matrix and $L'$ has odd colors.
Then  for $k=\max\{k_1,\dots, k_m\}$ we have
\[
J_{L\sqcup L'}(\tilde{P}'_{\bk},\bj) \in \frac{(q^{k+1};q)_{k+1}}{1-q}
\,\,\BZ[q^{\pm 1}].
\]
\end{AThm}

 In the case $L'=\emptyset$, this statement was proved
 in \cite[Theorem 8.2]{Ha}. Since our proof is a modification
of the original one, we first
 sketch Habiro's original proof for the reader's convenience.

\subsection{Sketch of  the proof of Habiro's  integrality theorem}

\subsection*{Geometric part}
Let us first recall
 the notion of a bottom
tangle, introduced by Habiro in \cite{Ha-b}.

An $n$--component bottom tangle $T=T_1\sqcup \dots \sqcup
T_n$ is a framed tangle consisting of $n$ arcs $T_1,\ldots ,T_n$ in a cube
such that all the endpoints of $T$ are on a line at the bottom square
of the cube, and for each $i=1,\ldots ,n$ the component $T_i$ runs from the
$2i$th endpoint on the bottom to the $(2i-1)$st endpoint on the
bottom, where the endpoints are counted from the left.
An example, the Borromean bottom tangle $B$, is given in Figure \ref{borro}.
\begin{figure}[ht]
\mbox{\epsfysize=2cm \epsffile{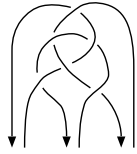}}
\caption{Borromean bottom tangle $B$}
\label{borro}
\end{figure}

In \cite{Ha-b}, Habiro defined a braided subcategory $\modB$ of the
category of framed, oriented tangles which acts on the
bottom tangles by composition (vertical pasting).
The objects of $\modB $ are the symbols $\modb ^{\otimes n}$, $n\ge 0$,
where $\modb:=\downarrow \uparrow $.
 For $m,n\ge 0$, a
morphism $X$ of $\modB$ from $\modb ^{\otimes m}$ to $\modb ^{\otimes n}$ is
the isotopy class of a
framed, oriented tangle $X$ which we can compose with $m$--component
bottom tangles to get $n$--component bottom tangles.
 Let
$\modB (m,n)$ be the set of morphisms from $\modb ^{\otimes m}$
to $\modb ^{\otimes n}$.
The composite $YX$ of two morphisms  is the gluing of $Y$
to the bottom of $X$, and the identity morphism
$1_{\modb ^{\otimes m}}=\downarrow \uparrow \cdots\downarrow \uparrow $
is a tangle consisting of $2m$ vertical
arcs.  The monoidal structure is
given by pasting  tangles side by
side.   The  braiding for the
generating object $\modb$
 with itself is given by
$\psi _{\modb ,\modb }=\incl{1.5em}{phibb+2}$.




 Corollary 9.13 in \cite{Ha-b} states the following.
\begin{Aproposition}\label{thmASL}{\em (Habiro)}
If the linking matrix of a bottom tangle $T$ is zero
then $T$ can be presented as $T=W B^{\otimes k}$, where $k\geq 0$ and $W\in \modB(3k,n)$ is
 obtained by
 horizontal and vertical pasting of finitely many copies of
$1_\modb$, $\psi_{\modb,\modb}$, $\psi^{-1}_{\modb,\modb}$,
 and

  \begin{equation*}
    \def\sss{1.5em}
    \eta _\modb =\incl{\sss}{bot0},\quad
    \def\sss{2em}
    \mu _\modb =\incl{\sss}{bomu},\quad
    \gamma_+=\incl{2.5em}{gamma+},\quad
    \gamma_-=\incl{2.5em}{gamma-}.
  \end{equation*}
\end{Aproposition}

\subsection*{Algebraic part}
Let  $K=v^{H}=e^\frac{hH}{2}$.
Habiro introduced the integral version
 $\U_q$, which is
the $\Z[q, q^{- 1}]$--subalgebra of $U_h$ freely
spanned over $\Z[q, q^{- 1}]$ by $\tilde{F}^{(i)}K^{j}e^k$
for $i,k\geq 0, j\in \Z$,
where
\[
\tilde{F}^{(n)}=\frac{F^nK^n}{v^{\frac{n(n-1)}{2}}[n]!}\;\;\;\;\text{ and
 }\;\;\;\; e=(v-v^{-1})E.
\]

There is $\Z/2\Z$--grading, $\U_q=\U^0_q\oplus \U^1_q$, where 
 $\U^{0}_q $ (resp. $\U_q^{1}$) is  spanned
 by $\tilde{F}^{(i)}K^{2j}e^k$ (resp.
$\tilde{F}^{(i)}K^{2j+1}e^k$). We call this the $\ve$--grading, and 
 $\U^{0}_q $ (resp. $\U_q^{1}$) the even (resp. odd) part.
 
The two--sided ideal $\F_p$ in $\U_q$ generated by $e^p$
 induces a filtration on $(\U_q)^{\otimes n}$, $n\geq 1$, by

$$ \F_p((\U_q^{})^{\otimes n})=
\sum^n_{i=1} (\U_q^{})^{\otimes i-1}\otimes
\F_p(\U_q^{})\otimes (\U_q^{})^{\otimes n-i}
\subset (\U_q^{})^{\otimes n}\, .$$
Let  $(\tilde{\U}_q)^{\tilde\otimes n}$ be
the image of the homomorphism
\[
\lim_{\overleftarrow{\hspace{1mm}{p\geq 0}\hspace{1mm}}} \;\;
\frac{(\U_q)^{\otimes n}}{\F_p((\U_q)^{\otimes n})} \to
U_h^{\hat \otimes n}
\]
where $\hat\otimes$ is the $h$--adically completed tensor product.
By using
$\F_p(\U_q^\ve):= \F_p(\U_q)\cap \U^{\ve}_q$ one defines
$(\tilde{\U}^{\ve}_q)^{\tilde\otimes n}$
for $\ve\in\{0,1\}$
in a similar fashion.

By definition (Section 4.2 of \cite{Ha}),
the universal $sl_2$ invariant  $J_T$ of an $n$--component
 bottom tangle $T$
 is an element of $U_h^{\hat\otimes n}$.
Theorem 4.1 in \cite{Ha} states  that, in fact,
for any bottom tangle $T$ with zero
linking matrix, $J_T$ is even, i.e.
\be\label{even}
J_T\in (\tilde \U^{0}_q)^{\tilde \otimes n}\, . \ee

Further, using the fact that $J_K$ of a $0$--framed bottom knot $K$
 (i.e. a 1--component bottom tangle) belongs to the center of $\tilde \U^0_q$, Habiro showed that
$$J_K=\sum_{n\geq 0}  (-1)^n q^{n(n+1)}\frac{(1-q)}{(q^{n+1};q)_{n+1}}\,  J_K(\tilde P'_n)\, \sigma_n$$
where
$$\sigma_n=\prod^n_{i=0}(C^2-(q^i+2+ q^{-i}))\, \quad {\rm with}\quad
C=(v-v^{-1})\tF^{(1)}K^{-1}e+vK+v^{-1}K^{-1}\, ,$$  the quantum Casimir
operator. The $\sigma_n$ provide a basis for the even part of the center.
From this, Habiro deduced that $J_K(\tilde P'_n)\in \frac{  (q^{n+1};q)_{n+1}}{(1-q)} \Z[q,q^{-1}]$.

The case of  $n$--component bottom tangles reduces to the 1--component case by partial trace, using certain integrality of traces of even element (Lemma 8.5 of \cite{Ha}) and the fact that
$J_T$ is invariant under the adjoint action.


\subsection*{Algebro--geometric part}
The proof of   \eqref{even}
uses  Proposition \ref{thmASL}, which allows to build
any bottom tangle $T$ with zero linking matrix from simple parts,
i.e. $T=W(B^{\otimes k})$.


On the other hand, the construction of the universal invariant $J_T$
extends to the braided functor $J:\modB\to \operatorname{ Mod}_{U_h}$
from $\modB$ to the category of $U_h$--modules.
This means that
$J_{W(B^{\otimes k})}=J_W(J_{B^{\otimes k}})$.
Therefore, in order to show \eqref{even}, we need to check
that $J_B \in (\tilde \U^0_q)^{\tilde \otimes 3} $,  and then verify that
$J_W$ maps the even part to itself.
The first check can be done by a direct computation
\cite[Section 4.3]{Ha}.
The last verification is the content
of Corollary 3.2 in \cite{Ha}.


\newcommand{\gr}{\operatorname{gr}}
\newcommand{\tcS}{\tilde {\mathcal S}}
\newcommand{\cS}{{\mathcal S}}

\subsection{Proof of Theorem \ref{main-integrality}
}

\subsection*{Generalization of Equation  \eqref{even}}
To prove Theorem \ref{main-integrality} we need a generalization 
of Equation \eqref{even} or Theorem 4.1 in \cite{Ha} to
 tangles with closed components.
To state the result let us first introduce two new gradings.

Suppose $T$ is an $n$--component bottom tangle in a cube, homeomorphic to
the 3--ball  $D^3$.
Let $\tcS(D^3\setminus T)$ be the $\Z[q^{\pm 1/4}]$--module freely generated
by the isotopy classes of
framed unoriented colored links in $D^3\setminus T$, including the empty
link. For such
a link $L \subset D^3\setminus T$ with $m$--components colored by
$n_1,\dots, n_m$, we define our new gradings  as follows.
First provide the components of $L$ with arbitrary orientations. 
Let $l_{ij}$ be the linking
number between the $i$th component of $T$ and the $j$th component of $L$, 
and $p_{ij}$
 be the linking number between the $i$th and the $j$th components of $L$.
For $X=T\sqcup L$ we put
\begin{equation}\gr_\ve(X):= (\ve_1,\dots, \ve_n)\in (\Z/2\Z)^n, \quad  \text{where}\quad
 \ve_i := \sum_{j} l_{ij} n'_{j} \pmod 2, \quad \text{ and}
\label{3306}
\end{equation}

$$ 
\gr_q(L) := \sum_{1\le i,j\le m }p_{ij} n_i'n_j' + 2\sum_{1\le j\le m}(p_{jj}+1) n_j' \pmod 4,
\quad \text{where} \quad n_i':= n_i-1.
$$
It is easy to see that the definitions do not depend on the orientation of $L$.

The meaning of $\gr_q(L)$ is the following: The colored Jones polynomial of $L$,
a priori  a Laurent polynomial of $q^{1/4}$,
is actually a Laurent polynomial of $q$ after dividing by
$q^{\gr_q(L)/4}$; see \cite{LeDuke} for this result and its generalization to other Lie algebras.

We further extend both gradings to $\tcS(D^3\setminus T)$ by
$$ 
\gr_\ve(q^{1/4})=0, \quad \gr_q(q^{1/4}) =1 \pmod 4.
$$

\vskip2mm

Recall that the universal invariant $J_X$ can also be defined when 
$X$ is the union of a bottom tangle and 
a colored link (see \cite[Section 7.3]{Ha-b}). 
In \cite{Ha-b}, it is proved that $J_X$ is adjoint invariant.
The generalization of
Theorem 4.1 of \cite{Ha} is the following.
 
\begin{AThm}\label{MMM}
Suppose $X=T \sqcup L$, where $T$ is a $n$--component bottom tangle with zero
linking matrix
and $L$ is a framed unoriented colored link with 
$\gr_\ve(X)= (\ve_1,\dots,\ve_n)$.
 Then
$$ J_X \in  q^{{\gr_q(L)/4}}\, \left ( \tilde \U_q^{\ve_1}\, \tilde \otimes
\dots \tilde \otimes\; \tilde \U_q^{\ve_n}\right ).$$
\label{3310}
\end{AThm}

\begin{Acorollary}\label{MMM1}
Suppose $L$ is colored by a tuple of odd numbers, then
$$J_X\in (\tilde \U_q^{0})^{\tilde\otimes n}\, .$$
\end{Acorollary}

Since $J_X$  is invariant under the adjoint action, Theorem
\ref{main-integrality} follows from Corollary \ref{MMM1} by repeating
Habiro's arguments.
\qed

Hence it remains to prove Theorem \ref{MMM}.
In the proof we will need a notion of a {\em good morphism}.

\subsubsection*{Good morphisms}
Let  $I_m:=1_{\modb^{\otimes m}}\in \modB(m,m)$ be the identity morphism of $\modb ^{\otimes m}$
in the cube $C$.
A framed link $L$ in the complement $C \setminus I_m$ is
{\em good} if $L$ is geometrically disjoint from all the
up arrows of $\modb ^{\otimes m}$, i.e.
there is a plane dividing the cube into two  halves, such that all  the up
arrows are in one half,
and all the down arrows and $L$ are in the other. Equivalently,
there is a diagram in which all
the up arrows are above all components of $L$. The union $W$ of $I_m$ and a
 colored framed good link $L$ is called a {\em good morphism}.
If $Y$ is any bottom tangle so that we can compose $X =WY$, then it is easy to see that 
$\gr_\ve(X)$ does not depend on $Y$, and we define 
$\gr_\ve(W):=\gr_\ve(X)$.   Also define $\gr_q(W):= \gr_q(L)$.

As in the case with $L=\emptyset$, the universal invariant extends
to a map $J_W: \U_h^{\otimes m} \to \U_h^{\otimes m}$.

\subsection*{Proof of Theorem \ref{MMM}}
The strategy here is again analogous to the Habiro case: 
In Proposition \ref{TWWB}
we will decompose $X$ into simple parts: the top is a bottom tangle
with zero linking matrix, the next is a good morphism, 
and the bottom is
a morphism obtained by pasting copies of $\mu_\modb$.
Since, any bottom tangle with zero linking matrix satisfies
Theorem \ref{MMM} and $\mu_\modb$ is the product in $\U_q$, which preserves
the gradings, it remains to show that 
any good morphism preserves the gradings.
This is done in
Proposition \ref{MM2} below.
\qed


\begin{Aproposition}\label{TWWB}
Assume $X=T\sqcup L$ where $T$ is a $n$--component bottom tangle with zero linking matrix
and $L$ is a link. Then
there is a presentation
$X =W_2 W_1 W_0$,
where $W_0$ is a bottom tangle with zero linking matrix, $W_1$ is a good morphism, and
 $W_2$
is obtained by pasting  copies of $\mu_\modb$.
\end{Aproposition}

\begin{proof}
 Let us first  define
 $\tilde \gamma_\pm\in \modB(i, i+1)$ for any $i\in {\mathbb N}$ as follows.
 \begin{equation*}
    \def\sss{3em}
  \tilde\gamma_+  =\incl{\sss}{tildegamma+}\quad
    \def\sss{3em}
    \tilde\gamma_- =\incl{\sss}{tildegamma-}.
    \end{equation*}
    
  If a copy of  $\mu_\modb$ is directly above $\psi^{\pm 1}_{\modb, \modb}$ or
$ \gamma_\pm$, one can move $\mu_\modb$ down by isotopy
and represent the result by pasting copies of $\psi^{\pm 1}_{\modb, \modb}$ and
$\tilde \gamma_\pm$.
It is easy to see that after the isotopy $\gamma_\pm$
gets replaced by $\tilde\gamma_\pm$ and $\psi^{\pm 1}_{\modb, \modb}$
by two copies of $\psi^{\pm 1}_{\modb, \modb}$. 
    
Using Proposition \ref{thmASL} and reordering the basic morphisms so that the $\mu$'s are at the bottom, one can
see that $T$ admits the following
presentation:
$$T= W_2 \tilde W_1 (B^{\otimes k})$$
where $B$ is the Borromean tangle, $W_2$ is obtained by pasting copies of
$\mu_\modb$ and
$\tilde{W}_1$ is obtained by pasting copies of $\psi^{\pm 1}_{\modb,\modb}$,
$\tilde\gamma_{\pm}$ and $\eta_\modb$.

Let $P$ be the horizontal plane separating
$\tilde W_1$ from $W_2$. Let $P_+$ ($P_-$) be the upper (lower, respectively)
half--space. Note that $W_0=\tilde W_1(B^{\otimes k})$ is a bottom tangle with zero linking matrix lying in $P_+$  and does not have
any minimum points. Hence the pair
$(P_+,W_0)$ is homeomorphic to the pair
$(P_+,\; l \text{ trivial arcs})$.
Similarly, $W_2$ does not have any maximum points; hence $L$
can be isotoped  off $P_-$ into $P_+$. Since the pair
$(P_+,W_0)$ is homeomorphic to the pair
$(P_+,\; l\text{ trivial arcs})$ one can isotope $L$ in $P_+$ to the bottom end points of down arrows. 
We then obtain the desired presentation.
\end{proof}

\begin{Aproposition}\label{MM2} For every good morphism $W$, the operator $J_{W}$ preserves gradings in the following sense.
 If $x\in \U_q^{\ve_1} \otimes \dots \otimes \U_q^{\ve_m}$,
then 
$$ 
J_{W}(x) \in q^{\gr_q(W)/4} \left( \U_q^{\ve_1'} \otimes \dots \otimes
 \U_q^{\ve_m'}\right), \quad \text{where} \quad (\ve'_1,\dots, \ve'_m) =
 (\ve_1,\dots, \ve_m) + \gr_{\ve}(W).
$$
\end{Aproposition}
The rest of the appendix is devoted to the proof of Proposition \ref{MM2}.

\subsection*{Proof of Proposition \ref{MM2}}\label{ProofMM2}
We proceed as follows. Since $J_X$ is invariant under cabling and
skein relations, and by Lemma  \ref{A8} below, both relations
preserve $\gr_\ve$ and $\gr_q$, we consider the quotient
of $\tilde S(D^3\setminus T)$ by these relations known as a skein module
of $D^3\setminus T$. 
For  $T=I_n$, this module has a natural algebra structure, with
good morphisms forming a subalgebra.
By Lemma \ref{A9} 
(see also Figure \ref{Wgamma}),
the basis elements $W_\gamma$ of this subalgebra are labeled by $n$--tuples
$\gamma=(\gamma_1,\dots, \gamma_n)\in (\Z/2\Z)^n$.
 It's clear that if the
proposition holds for $W_{\gamma_1}$ and $W_{\gamma_2}$,
then it holds for $W_{\gamma_1} W_{\gamma_2}$. 
Hence 
it remains to check the claim for $W_\gamma$'s.
This is done in Corollary \ref{corA10} 
for basic good morphisms corresponding to $\gamma$ whose
  non--zero $\gamma_j$'s are consecutive.
 Finally,  any $W_\gamma$ can be 
obtained by pasting a basic good morphism  with few copies of $\psi^{\pm}_{\modb, \modb}$.
Since $J_{\psi^\pm}$ preserves gradings (compare
(3.15), (3.16) in \cite{Ha}), the claim follows
from Lemmas \ref{A9}, \ref{A8} and Proposition \ref{3308} below.
\qed

\subsubsection*{Cabling and skein relations}
Let us introduce the following relations in $\tcS(D^3\setminus T)$.

Cabling relations:
\begin{itemize}
\item[(a)] Suppose $n_i=1$ for some $i$. The first
cabling relation is $L=\tilde L$, where $\tilde L$ is obtained from $L$ 
by removing the $i$th component.
\item[(b)] Suppose $n_i \ge 3$ for some $i$. The second cabling relation 
is $L= L'' -L'$,
where $L'$
is the link $L$ with the color of the $i$th component switched to $n_i-2$, 
and $L''$
is obtained
from $L$ by replacing the $i$th component with two of its parallels, which are
colored with $n_i-1$ and $2$.
\end{itemize}

Skein relations:
\begin{itemize}
\item[(a)] The first skein relation is $U=q^{\frac 12}+q^{-\frac 12}$, where $U$ denotes the
unknot with framing zero and color 2.
\item[(b)] Let $L_R$, $L_V$ and $L_H$ be unoriented  framed links with colors 2 which are
identical except in a disc where they are as shown in Figure \ref{Skein}. Then the second skein
relation is $L_R = q^{\frac 14}L_V + q^{-\frac 14}L_H$ if the two strands in the crossing come 
from different components of $L_R$, and $L_R = \epsilon(q^{\frac 14}L_V - q^{-\frac 14}L_H)$ 
if the two strands come from the same component of $L_R$, producing a crossing of sign 
$\epsilon=\pm 1$ (i.e. appearing as in $L_{\epsilon}$ of Figure \ref{Skein} if $L_R$ is oriented).
\end{itemize}
\begin{figure}[ht]
\mbox{\hspace{1.3cm}\epsfysize=1.3cm
\epsffile{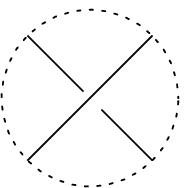}\hspace{7mm}
\epsfysize=1.3cm\epsffile{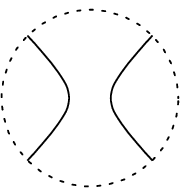}\hspace{7mm}
\epsfysize=1.3cm\epsffile{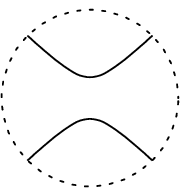}\hspace{7mm}
\epsfysize=1.3cm\epsffile{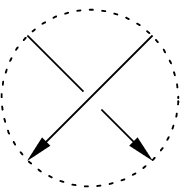}\hspace{7mm}
\epsfysize=1.3cm\epsffile{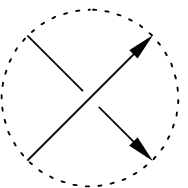}\hspace{7mm}
}
\caption{\hspace{0mm} $L_R$ \hspace{1.5cm} $L_V$ \hspace{1.5cm} 
$L_H$ \hspace{1.5cm} $L_{+}$ \hspace{1.5cm} $L_{-}$
}
\label{Skein}
\end{figure}

We denote by $S(D^3\setminus T)$
the quotient of $\tcS(D^3\setminus T)$ by these relations. It is known as
the  {\it skein module} of $D^3\setminus T$ (compare   \cite{Pr},
\cite{SikPr} and \cite{Bullock}).   Recall
that the ground ring is $\Z[q^{\pm 1/4}]$.

Using the cabling relations, we can reduce all colors of $L$ in 
$S(D^3\setminus T)$ to be 2. 
Note that the skein module
 $S(C\setminus I_n)$
has a natural algebra structure, given by putting one cube on the
 top of the other.
Let us denote by $A_n$
the subalgebra of this skein algebra generated by good morphisms.


For a set $\gamma=(\gamma_1,\dots, \gamma_n)\in (\Z/2\Z)^n$
 let $W_\gamma$ be
 a simple closed  curve  encircling the end points of those downward  arrows with 
 $\gamma_i=1$. See Figure \ref{Wgamma} for an example.
 \begin{figure}[ht!]
\mbox{
\epsfysize=1.7cm\epsffile{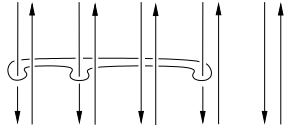}
}
\caption{\hspace{0mm} The element $W_{(1,1,0,1,0)}$.
\label{Wgamma}}
\end{figure}

 
 Similarly to the case of Kauffman bracket skein module \cite{Bullock}, one can easily prove 
 the following.
\begin{Alemma}  The algebra $\mathcal A_n$ is generated by $2^n$  curves $W_\gamma$.
\label{A9}
\end{Alemma}

Using linearity, we can extend the definition of $J_X$
to $X =T \sqcup L$, where $L$ is {\em any element}
of $\tcS(D^3\setminus T)$. It is known that $J_X$ 
 is invariant under the cablings and skein relations 
(Theorem 4.3 of  \cite{KM}), hence $J_X$ 
 is defined for $L \in S(D^3\setminus T)$.
Moreover, we have

\begin{Alemma} \label{A8}
Both gradings $\gr_\ve$ and $\gr_q$ are preserved under the cabling
and skein relations.
\end{Alemma}

\begin{proof}
The statement is obvious for the $\ve$--grading.
For the
$q$--grading, notice that 
\[
\gr_q(L)=2\sum_{1\leq i<j\leq m}p_{ij}n_i' n_j' +
\sum_{1\leq j \leq m}p_{jj}n_j'^2+2\sum_{1\leq j\leq m}(p_{jj}+1)n_j',
\]
and therefore $\gr_q(L'')\equiv \gr_q(L')\equiv\gr_q(L)\pmod{4}$.
This takes care of the cabling relations. 

Let us now assume that all colors of $L$ are equal to 2 and therefore
\[
\gr_q(L)=2\sum_{1\leq i< j\leq m} p_{ij}+3\sum_{i=1}^{m} p_{ii}+2m.
\]
The statement is obvious for the first skein relation.
For the second skein relation, choose an arbitrary orientation on L. Let us first assume that 
the two strands in the crossing depicted in Figure \ref{Skein} come from the same component 
of $L_R$ and that the crossing is positive. Then, $L_V$ and $L_H$ have one positive 
self--crossing less, and $L_V$ has one link component more than $L_R$. Therefore
\begin{eqnarray*}
\gr_q(q^{\frac 14}L_V)&=&\gr_q(L_R)-3+2+1 \equiv \gr_q{L_R} \pmod{4}, \\
\gr_q(q^{-\frac 14}L_H)&=&\gr_q(L_R)-3-1 \equiv \gr_q{L_R} \pmod{4}.
\end{eqnarray*}
It is obvious, that this does not depend on the orientation of $L_R$. If the crossing 
of $L_R$ is negative or the two strands do not belong to the same component of $L_R$, 
the proof works similar.
\end{proof}

\subsubsection*{Basic good morphisms}
Let  $ Z_n$ be $W_\gamma$ for $\gamma=(1,1,\dots, 1)\in (\Z/2\Z)^n$.

\begin{equation*}\label{Zn}
    \def\sss{3em}
	  Z_n  =\incl{\sss}{ZnVers2}\quad\quad
\end{equation*}

\begin{Aproposition} One has a presentation
$$ J_{ Z_n} = 
\sum z^{(n)}_{i_1} \otimes \sum z^{(n)}_{i_2} \otimes \dots \otimes
\sum z^{(n)}_{i_{2n}},$$
such that  $z^{(n)}_{i_{2j-1}}\, z^{(n)}_{i_{2j}}\in v\; \U^1_q$
 for every $j=1,\dots, n$.
\label{3308}
\end{Aproposition}

\begin{Acorollary}\label{corA10}
$J_{ Z_n}$ satisfies Proposition
\ref{MM2}.
\end{Acorollary}

\begin{proof}
Assume  $x\in \U_q^{\ve_1} \otimes \dots \otimes \U_q^{\ve_n}$,
then  we have
$$ J_{ Z_n}(x) = 
\sum z^{(n)}_{i_1} x_1  z^{(n)}_{i_2} \otimes \dots \otimes
\sum z^{(n)}_{i_{2n-1}} x_n z^{(n)}_{i_{2n}}.$$
Hence, by Proposition \ref{3308}, we get
$$ 
J_{ Z_n}(x) \in q^{1/2} \left( \U_q^{\ve_1'} \otimes \dots \otimes
 \U_q^{\ve_n'}\right), \quad \text{where} \quad (\ve'_1,\dots, \ve'_m) =
 (\ve_1,\dots, \ve_n) + (1,1,\dots, 1).
$$
The claim follows now from the fact that
$\gr_\ve( Z_n)=(1,1,\dots,1)$ and $\gr_q(L)=2$.
\end{proof}
\subsection{Proof of Proposition \ref{3308}}
 The statement holds true for $J_{ Z_1}= C\otimes \id_\uparrow$. 
Now Lemma 7.4 in \cite{Ha-b} states that
 applying
 $\Delta$ to the $i$th component of 
the universal quantum invariant of a tangle 
is the same 
as duplicating the $i$th component.
Using this fact we represent
\be\label{Jn} J_{Z_{n+1}} = (1_{\modb^{\otimes n-1}}\otimes  {\Phi}
\otimes \id_\uparrow) \left(  J_{Z_{n}} \right),\ee
where $\Phi$ is defined as follows. 
For $x \in \U_q$ 
with  $\Delta (x)=\sum x_{(1)}\otimes x_{(2)}$, we put
\[
\Phi(x) :=   \sum_{(x), m,n} x_{(1)} \otimes \beta_m S(\beta_n)
\otimes \alpha_n \, x_{(2)} \alpha_m
\]
where the $R$--matrix is given by $R =\sum_l {\alpha_l \otimes \beta_l}$.
See Figure below for a picture. 
\begin{figure}[!h]
\mbox{
\input{Phi.pstex_t}}
\end{figure}

We are left with the computation of the 
$\ve$--grading of each component of $\Phi(x)$.

In $\U_q$, in addition to the $\ve$--grading,
 there is also the $K$--grading, defined by
$|K|=|K^{-1}|=0, |e|=1, |F|=-1$. In general, the co--product $\Delta$ does not
preserve the $\ve$--grading. However, we have the following.
\begin{Alemma} Suppose $x\in \U_q$ is homogeneous in both $\ve$--grading and $K$--grading. 
Then we have a presentation
$$
\Delta(x) = \sum_{(x)} x_{(1)} \otimes x_{(2)},
$$
where each $x_{(1)}$, $x_{(2)}$ are homogeneous with respect to the
$\ve$--grading and $K$--grading. In addition,
for $x\in \U_q^\ve$, we have $x_{(2)}\in \U_q^\ve$ and
$x_{(1)} \, K^{-|x_{(2)}|}\in \U_q^\ve$.

\label{3304}
\end{Alemma}

\begin{proof} If the statements hold true for $x, y \in \U_q$, then they hold
true for $xy$.
Therefore, it is enough to check the statements for the generators $e, \tF^{(1)},$ and $K$, for which
they follow from explicit formulas of the co--product.
\end{proof}

\begin{Alemma} Suppose $x \in \U_q$ is homogeneous in both $\ve$--grading and
$K$--grading. 
There is a presentation
$$
\Phi(x) =\sum x_{i_0}\otimes x_{i_1}\otimes  x_{i_2}
$$
such that each $x_{i_j}$ is homogeneous in both $\ve$--grading and
$K$--grading, and for $x \in \U^\ve_q$, 
$x_{i_2}$ and 
 $ x_{i_0}\, x_{i_1}$ belong to $\U^\ve_q$.
\label{3303}
\end{Alemma}
\begin{proof}
We put $D=\sum D' \otimes D'':=v^{\frac 12 H\otimes H}$.
Using (see e.g. \cite{Ha}) 
\[
R=D\left(\sum_{n} q^{\frac 12 n(n-1)}\tF^{(n)}K^{-n}\otimes e^n\right),
\]
we get
\begin{eqnarray*}
\Phi(x)&=&\sum_{(x),n,m}
q^{\frac{1}{2}\left(m(m-1)+n(n-1)\right)}
x_{(1)}\otimes
D''_2 e^m S(D''_1e^{n})\otimes
D'_1\tF^{(n)}K^{-n}x_{(2)}D'_2\tF^{(m)}K^{-m}\\
&=&
\sum_{(x),n,m}(-1)^n q^{-\frac 12 m(m+1)-n(|x_{(2)}|+1)}
x_{(1)}\otimes e^{m}e^n K^{-|x_{(2)}|}\otimes \tF^{(n)}x_{(2)}\tF^{(m)},
\end{eqnarray*}
where
we used $(\id \otimes S)D=D^{-1}$ and $D^{\pm1}(1\otimes x)=(K^{\pm|x|}\otimes x)D^{\pm 1}$ for
homogeneous $x\in \U_q$ with respect to the $K$--grading.
Now, the claim follows from Lemma \ref{3304}.
\end{proof}
 By induction on $n$ in \eqref{Zn}, given that
$C\in v \;\U^1_q$,
Lemma \ref{3303} implies Proposition \ref{3308}.
\qed



\begin{thebibliography}{[EMSS]}


\bibitem{And} G. Andrews, {\em  q--series: their
development and applications in analysis, number theory,
combinatorics, physics, and computer algebra}, regional conference
series in mathematics, Amer. Math. Soc. {\bf 66} (1985)


\bibitem{BBL} A. Beliakova, C. Blanchet, T. T. Q. Le,
{\em  Unified quantum invariants and their refinements
for homology 3-spheres with 2-torsion}, Fund. Math. {\bf 201} (2008)
217--239


\bibitem{BL} A. Beliakova, T. T. Q.  Le, {\em Integrality of quantum
3--manifold invariants and rational surgery formula},
{Compositio Math.} {\bf 143} (2007) 1593--1612


\bibitem{Bullock} D. Bullock,
{\em A finite set of generators for the Kauffman bracket skein algebra},
Mathematische Zeitschrift {\bf 231} (1999)

\bibitem{GR} G. Gasper, M. Rahman, {\em Basic Hypergeometric Series},
Encyclopedia  Math. {\bf 35} (1990)


\bibitem{GM} P. Gilmer, G. Masbaum, {\em Integral lattices in TQFT},
Ann. Scient. Ecole Normale Sup.  {\bf 40}   (2007) 815--844



\bibitem{Ha} K. Habiro,
{\em A unified Witten--Reshetikhin--Turaev invariant for integral homology
spheres}, Invent. Math. {\bf 171}  (2008) 1--81

\bibitem{Ha-b} K. Habiro, {\em Bottom tangles and universal invariants},
Algebraic $\&$ Geometric Topology  {\bf 6}  (2006)   1113--1214


\bibitem{Ha1} K. Habiro, {\em Cyclotomic completions of polynomial rings},
Publ. Res. Inst. Math. Sci.  {\bf 40}  (2004)   1127--1146





\bibitem{KK} A. Kawauchi, S. Kojima, {\it Algebraic classification of linking parings
on 3-manifolds}, Math. Ann. {\bf 253} (1980) 29--42

\bibitem{Kho} M. Khovanov, {\em Hopfological algebra and categorification
at a root of unity: the first steps}, math.QA/0509083 (2005), to appear in
Commun.  Contemp. Math.


\bibitem{KM-Dedekind} R. Kirby, P. Melvin, {\em Dedekind sums,
$\mu$--invariants
and the signature cocycle}, Math. Ann. {\bf 299} (1994)  231--267

\bibitem{KM} R. Kirby, P. Melvin, {\em The $3$--manifold invariants of Witten and Reshetikhin--Turaev for ${\rm sl}(2,C)$},
Invent. Math. {\bf 105} (1991)  473--545


\bibitem{Lang} S. Lang, {\em Algebra}, 3rd edition, Addison-Wesley 1997.

\bibitem{Le} T. T. Q. Le, {\em Strong integrality of quantum
    invariants of 3--manifolds}, Trans. Amer. Math. Soc. {\bf 360} (2008)
 2941--2963

\bibitem{Le3} T. T. Q. Le, {\em Quantum invariants of 3-manifolds: integrality,
splitting, and perturbative expansion},  Topology Appl.  {\bf 127}
(2003)  125--152


\bibitem{Le2} T. T. Q. Le, {\it On the perturbative $PSU(n)$ invariants of rational homology
3-spheres}, Topology {\bf 39} (2000), no. 4, 813--849


\bibitem{LeDuke} T. T. Q. Le, {\em Integrality and symmetry of quantum
link invariants},  Duke Math. Journal  {\bf 102}
(2000)  273--306

\bibitem{LMO} T. T. Q. Le, J. Murakami, T. Ohtsuki,
{\em On a universal perturbative
invariante of 3-manifolds}, Topology {\bf 37} (1998) 539--574


\bibitem{LiLi}
B.--H. Li, T.--J.  Li, {\em  Generalized Gaussian sums:
Chern--Simons--Witten--Jones invariants of lens spaces},  J. Knot
Theory Ramif. {\bf  5} (1996) 183--224



\bibitem{MR} G. Masbaum, J. Roberts,
{\it A simple proof of integrality of quantum invariants at prime
roots of unity}, Math. Proc. Camb. Phil.  Soc. {\bf 121} (1997),
443--454


\bibitem{MM} P. Melvin, H. Morton,
{\it The coloured Jones function},
Comm. Math. Phys. {\bf 169} (1995)  501--520




\bibitem{Mu}  H. Murakami,
{\it Quantum $SO(3)$--invariants dominate the $SU(2)$--invariant of
Casson and Walker}, Math. Proc. Camb. Phil.  Soc. {\bf 117} (1995)
237--249

\bibitem{Na} T. Nagell, {\it Introduction to Number Theory}, Almqvist \& Wiksells Boktryckeri (1951)

\bibitem{Ohtsukibook} T. Ohtsuki, {\em Quantum invariants.
A study of knots, 3-manifolds, and their sets}, Series on Knots and
Everything, World Scientific  {\bf 29}
(2002)


\bibitem{Oh} T. Ohtsuki, {\em A polynomial invariant of rational
homology $3$--spheres},  Invent. Math.  {\bf 123}  (1996) 241--257


\bibitem{Pr} J. H. Przytycki, {\em Skein modules of 3--manifolds}, Bull. Polish Acad.
Science, {\bf 39} (1--2) (1991) 91--100

\bibitem{SikPr} J. H. Przytycki, A. Sikora,
{\em On skein algebras and $Sl_2(\mathbb{C})$--character varieties},
Topology {\bf 39} (2000) 115--148


\bibitem{RG} H. Rademacher, E. Grosswald, {\em Dedekind Sums},
AMS Washington D. C.  (1972)



\bibitem{Tu} V. Turaev, {\em Quantum invariants of knots and
    3--manifolds}, de Gruyter Studies in Math. 18 (1994)

\bibitem{Wall} C. T. C. Wall,
{\em Quadratic forms on finite groups, and related topics},
Topology  {\bf 2}  (1963)  281--298

\end{thebibliography}
\end{document}